\documentclass{article}

\PassOptionsToPackage{numbers, compress}{natbib}
\usepackage[final]{neurips_2022}
\usepackage{todonotes}
\usepackage{booktabs}
\usepackage{mathtools}

\everypar{\looseness=-1}

\usepackage{hyperref}       
\usepackage{lipsum}
\usepackage{amsmath,amsthm,amssymb}
\usepackage{graphicx}
\usepackage{hhline,url}
\usepackage{mathrsfs}
\usepackage{xcolor}

\usepackage{algorithm,algpseudocode}

\algnewcommand{\Initialize}[1]{%
  \State \textbf{Initialize:}
  \State \hspace*{\algorithmicindent}\parbox[t]{0.8\linewidth}{\raggedright #1}
}

\newtheorem{thm}{Theorem}[]
\newtheorem*{thm*}{Theorem}
\newtheorem{m-thm}[thm]{Meta-Theorem}
\newtheorem*{m-thm*}{Meta-Theorem}
\newtheorem{lem}{Lemma}[]
\newtheorem{remark}{Remark}[]

\newtheorem{prop}{Proposition}[]
\newtheorem*{prop*}{Proposition}
\newtheorem{Definition}{Definition}

\newtheorem{Corollary}[thm]{Corollary}
\newenvironment{cor}{\begin{Corollary}}{\end{Corollary}}
\newtheorem{Example}[thm]{Example}

\newtheorem{algor}[thm]{Method}

\newtheorem{Condition}[thm]{Condition}

\newtheorem{assp}{Assumption}[]

\renewcommand{\H}{\mathcal{H}}

\newcommand{\R}{\mathbb{R}}
\newcommand{\dd}{\,\mathrm{d}}
\newcommand{\ve}{\varepsilon}

\newcommand{\F}{\mathcal{F}}

\renewcommand{\phi}{\varphi}

\newcommand{\X}{\mathcal{X}}
\newcommand{\bm}[1]{{\mbox{\boldmath $#1$}}}

\DeclareMathOperator{\cv}{conv}

\DeclareMathOperator{\supp}{supp}

\newcommand{\lmid}{\,\middle|\,}

\newcommand{\E}[1]{\mathbb{E}\!\left[#1\right]}
\renewcommand{\P}[1]{\mathbb{P}\!\left(#1\right)}
\newcommand{\ord}[1]{\mathcal{O}\!\left(#1\right)}
\newcommand{\K}{\mathcal{K}}

\newcommand{\ip}[1]{\left\langle #1 \right\rangle}

\newcommand{\vertiii}[1]{{\left\vert\kern-0.25ex\left\vert\kern-0.25ex\left\vert #1 
    \right\vert\kern-0.25ex\right\vert\kern-0.25ex\right\vert}}

\usepackage{caption}
\usepackage{subcaption}
\usepackage[inline]{enumitem}
\renewcommand{\tilde}{\widetilde}
\DeclareMathOperator{\wce}{wce}
\usepackage{bbm}

\usepackage{comment}

\newcommand{\Hil}{\mathcal{H}}

\usepackage[utf8]{inputenc} 
\usepackage[T1]{fontenc}    
\usepackage{url}            
\usepackage{booktabs}       
\usepackage{amsfonts}       
\usepackage{nicefrac}       
\usepackage{microtype}      
\usepackage{xcolor}         

\title{Positively Weighted Kernel Quadrature via Subsampling}

\author{%
	Satoshi Hayakawa, \ Harald Oberhauser, \ 
	Terry Lyons\\
	Mathematical Institute,
	University of Oxford\\
	\texttt{\{hayakawa,oberhauser,tlyons\}@maths.ox.ac.uk}
}

\begin{document}
	\maketitle
	
	\begin{abstract}
		We study kernel quadrature rules with convex weights.
		Our approach combines the spectral properties of the kernel with recombination results about point measures.
		This results in effective algorithms that construct convex quadrature rules using only access to i.i.d.~samples from the underlying measure and evaluation of the kernel and that result in a small worst-case error.
		In addition to our theoretical results and the benefits resulting from convex weights, our experiments indicate that this construction can compete with the optimal bounds in well-known examples.
		\footnote{Code: \url{https://github.com/satoshi-hayakawa/kernel-quadrature}}
	\end{abstract}
	
	\everypar{\looseness=-1}
	\section{Introduction}
	The goal of numerical quadrature is to provide, for a given probability measure $\mu$ on a space $\X$, a set of points $x_1,\ldots,x_n\in \X$ and weights $w_1,\ldots,w_n\in \R$ such that 
	\begin{equation}\label{eq:integral approximation}
		\sum_{i=1}^n w_if(x_i) \approx \int_\X f(x)\dd\mu(x)
	\end{equation}
	holds for a large class of functions $f:\X \to \R$.
	Kernel quadrature focuses on the case when the function class forms a reproducing kernel Hilbert space (RKHS). 
	What makes kernel quadrature attractive, is that the kernel choice provides a simple and flexible way to encode the regularity properties of a function class. 
	Exploiting such regularity properties is essential when the integration domain $\X$ is high-dimensional or the function class is large.
	Additionally, the domain $\X$ does not have to be Euclidean, but can be any topological space that carries a positive semi-definite  kernel.
	
	More formally, given a set $Q  \coloneqq  \{(w_i,x_i):i=1,\ldots,n\}\subset \R \times \X$ denote with $\mu^Q  \coloneqq  \sum_{i=1}^n w_i \delta_{x_i}$ the resulting measure on $\X$.
	We refer to $Q$ [resp.~$\mu^Q$] as a quadrature [resp.~quadrature measure], to the points $x_1,\ldots,x_n$ as the support of $Q$ [resp.~$\mu^Q$].
	The aim of kernel quadrature is to construct quadrature measures $\mu^Q$ that have a small worst-case error
	\begin{align}
		\wce(Q;\Hil_k, \mu) & \coloneqq 
		\sup_{\lVert f\rVert_{\Hil_k}\le 1}\left\lvert
		\int_\X f(x) \dd\mu^Q(x)
		- \int_\X f(x)\dd\mu(x)
		\right\rvert 
		\label{wce},
	\end{align}
	where $\Hil_k$ denotes the RKHS associated with a positive semi-definite kernel $k$.
	If the weights are positive and sum up to one, $w_i >0$, $\sum_{i=1}^n w_i=1$, then we refer to $Q$ as a convex quadrature rule.
	
	\paragraph{Contribution.}
	The primary contribution of this article is to leverage recombination (a consequence of Carath\'eodory's Theorem) with spectral analysis of kernels to construct convex kernel quadrature rules and derive convergence rates.
	We also provide efficient algorithms that compute these quadrature rules; they only need access to i.i.d.~samples from $\mu$ and the evaluation of the kernel $k$.
	See Table \ref{table:comp} for a comparision with other kernel quadrature constructions.
	
	The table is written by using $\sigma_n$ and $r_n$,
	which represent a sort of decay of the kernel with respect to $\mu$.
	Typical regimes are
	$\sigma_n\sim n^{-\beta}$ (e.g. Sobolev) or $\sigma_n \sim \exp(-\gamma n)$
	(e.g. Gaussian) depending on the `smoothness' of the kernel \citep[e.g.][]{fas12,bac17} (see also Section \ref{sec:eig-gauss}),
	and in such regimes (with $\beta\ge2$ or $\gamma>0$),
	$\sigma_n$ or $r_n (\lesssim n\sigma_n)$ provide faster rates than
	$\wce^2 \sim 1/n$
	of the usual Monte Carlo rate.
	For more examples including multivariate Sobolev spaces,
	see \citet[][Section 2.3]{bac17}.
	
	\paragraph{Limitation.}
	Our proposed methods are based on either Mercer or Nystr{\"o}m approximation.
	Though our Mercer-based methods result in strong theoretical bounds, they require the knowledge of Mercer decomposition like \citep{bac17,bel19,bel20}, which is not available for general $(k, \mu)$.
	Our Nystr{\"o}m-based methods apply to much more general situations and outperform existing methods in experiments, but the $n/\sqrt{\ell}$ term makes their theoretical bound far from competitive.
	Further study is needed to bridge the gap between theory and empirical results.
	
	\begin{table}[h]
		\centering
		\renewcommand{\arraystretch}{1.2}
		\begin{tabular}{cccccc}
			Method & 
			\begin{tabular}{c}
				Bound of squared $\wce$
			\end{tabular}
			& Computational complexity &
			C & M & E
			\\\hline
			Herding \citep{che10,bac12} & $1/n$ & 
			$n\cdot(\text{$n$: global
				optimization})$
			& $\checkmark$ & $\checkmark$ & \\
			SBQ \citep{hus12} & Not found & 
			$n \cdot (\text{$n^2$:
				global optimization})$
			& & $\checkmark$ & \\
			Leveraged \citep{bac17} & 
			$\sigma_m$, $m =\ord{n\log n}$
			& Unavailable & & & \\
			DPP \citep{bel19,bel21} & $r_{n+1}$ &
			$n^3\cdot(\text{rejection sampling})$
			& & & \\
			CVS \citep{bel20}  & $\sigma_{n+1}$ &
			Unavailable & & $\checkmark$ & \\
			KT++ \citep{dwi21,dwi22,she22}  & $(1/n^2 + 1/N)\mathop\mathrm{polylog}(N)$ &
			$N\log^3N$ & $\checkmark$ & $\checkmark$ & $\checkmark$
			\\
			\hline
			
			\multicolumn{6}{l}{Ours:}\\\hline
			
			Mercer$^\dagger$ & $r_n$ &
			$nN_\phi + C(n, N_\phi)$
			& $\checkmark$
			&  & \\
			M. + empirical$^\ddagger$ & $r_n + \frac1N$ &
			$nN + n^3\log(N/n)$ & $\checkmark$
			&  & $\checkmark$ \\
			Nystr{\"o}m$^\dagger$ &
			$n\sigma_n + r_{n+1}
			+ \frac{n}{\sqrt{\ell}}$
			&
			$n\ell N_\phi + n\ell^2 + C(n, N_\phi)$
			&
			$\checkmark$ & $\checkmark$ &
			\\
			N. + empirical$^\ddagger$ &
			$n\sigma_n + r_{n+1}
			+ \frac{n}{\sqrt{\ell}} + \frac1N$
			& $n\ell N + n\ell^2 + n^3\log(N/n)$  &
			$\checkmark$ & $\checkmark$ & $\checkmark$
			\\\hline
		\end{tabular}\renewcommand{\arraystretch}{1.0}
		\vskip\baselineskip
		\caption{Comparison on $n$-point kernel quadrature rules.
			We are omitting the $\mathcal{O}$ notation
			throughout the table.
			Note that the assumption under which
			the theoretical guarantee holds varies from
			method to method,
			and this table
			displays just a representative bound
			derived in the cited references.
			Here are remarks on the notation.
			(1) $\sigma_m$ is the $m$-th eigenvalue of
			the integral operator $\K$,
			and $r_m = \sum_{i = m}^\infty\sigma_i$.
			(2) The symbols in the first line respectively mean
			C: convex, M: {\it not} using the knowledge of Mercer decomposition,
			and E: {\it not} using the knowledge of
			expectations such as $\int_\X k(x, y)\dd\mu(y)$.
			(3) The $(\text{$m$: global optimization})$
			is indicating
			the cost of globally optimizing a function
			whose evaluation costs $\Theta(m)$.
			($\dagger$) Mercer/Nystr{\"o}m are the algorithms
			based on random convex hulls, see Section~\ref{sec:main-rchq} and
			Appendix~\ref{app:known expectations}.
			($\ddagger$) M./N. + empirical are the algorithms
			discussed in the main text. 
		}\label{table:comp}
	\end{table}
	
	\paragraph{Why Convex Weights?} There are several reasons why convex weights are preferable:
	\begin{enumerate*}[label=(\roman*)]
		\item \textbf{Positive Integral Operator:} Kernel quadrature provides an approximation of the integration operator $f \mapsto I(f)=\int f(x)\dd\mu(x)$.
		Hence, a natural requirement is to preserve basic properties of this operator and positive weights preserve the positivity of this operator.
		\item \textbf{Uniform estimates and Robustness:}
		In applications, the RKHS $\mathcal{H}_k$ may be mis-specified if a quadrature rule with negative weights is applied to a function $f \notin \mathcal{H}_k$, the approximation error \eqref{eq:integral approximation} can get arbitrary bad; in contrast, a simple estimate shows that convex weights give uniform bounds, see Appendix \ref{app:robustness}. 
		\item \textbf{Iteration:}
		Consider the $m$-fold product of quadrature formulas
		for approximating $\mu^{\otimes m}$ on $\X^m$.
		This is a common construction for a multidimensional quadrature formulas (e.g., for polynomials) from one-dimensional formulas \citep{tch15} or numerics for stochastic differential equations \citep{lyo04}.
		In doing so, working with a probability measure is strongly preferred, since otherwise the total variation of their $m$-fold product gets exponentially large as $m$ increases
		($\lVert \mu^{\otimes m}\rVert_{\mathrm{TV}} = \lVert \mu\rVert_{\mathrm{TV}}^m$).
	\end{enumerate*}

	\paragraph{Related Literature.}
	Roughly speaking, there have been two approaches to kernel-based quadrature formulas: kernel herding and random sampling.
	In kernel herding
	or its variants, the points $(x_i)_{i=1}^n$ are found iteratively, typically based on the Frank--Wolfe gradient descent algorithm \citep{che10,bac12,hus12}.
	
	In the random sampling approach, $(x_i)_{i=1}^n$ are sampled and subsequently the weights are optimized.
	Generically, this results only in a signed measure $\mu^Q$ but not a probability measure.
	\citet{bac17} and \citet{bel19} use the eigenvalues and the eigenfunctions of the integral operator $\K: f\mapsto\int_\X k(\cdot, y)f(y)\dd\mu(y)$ to obtain a Mercer-type decomposition of $k$ \citep{ste12}.
	\citet{bac17} then uses the eigenvalues and eigenfunctions of $\K$ to define an optimized measure from which the points $(x_i)$ are i.i.d.~sampled.
	This achieves a near optimal rate, but the exact sampling from this measure is
	usually unavailable, although for special cases, it can be done efficiently. 
	In contrast, \citet{bel19} proposes non-independent sampling based on the determinantal point process \citep[DPP;][]{hou06}.
	These two papers also treat the more general quadrature problem that includes a weight function $g\in L^2(\mu)$,
	i.e., approximating $\int_\X f(x)g(x)\dd\mu(x)$ for $f\in\Hil_k \subset L^2(\mu)$, which we do not discuss in this paper.
	Another recently introduced method is kernel thinning
	\citep{dwi21,dwi22},
	which aims at efficient compression of empirical measures that can be obtained by sampling like our `+ empirical' methods.
	Its acceleration \citep{she22} makes it a competitive candidate in terms of compressing	$N\sim n^2$ points (`KT++' in Table \ref{table:comp}).
	
	Finally, we emphasize that the kernel quadrature literature is vast, and the distinction between herding and sampling is only a rough dichotomy, see e.g.~\citep{de2005near,liu2017black,bri15,kan16,oat17,kar18,kan20,sou20}.
	Beyond kernel quadrature, our algorithms can also contribute to the density estimation approach in \citep{tur21} which relies on recombination based on Fourier features although we do not pursue this further in this article.
	
	\paragraph{Outline.}
	{Section \ref{sec:main} contains our main theoretical and methodological contribution.
		Section \ref{sec:num} provides numerical experiments on common benchmarks.
		The Appendix contains several extensions of our main result, proofs, and further experiments and benchmarks.}

	\section{Main Result}\label{sec:main}
	Assume we are given a set\footnote{The number $n-1$ stems from
		Carath{\'e}odory's theorem, Remark \ref{rem:mar}, and leads to an $n$ point quadrature rule.} of $n-1$ functions $\varphi_1,\ldots,\varphi_{n-1} :\X \to \R$ such that their linear combinations well approximate functions in $\Hil_k$.
	Then our kernel quadrature problem reduces to the construction of an $n$-point
	discrete probability measure $\mu^{Q_n}= \sum_{i=1}^n w_i \delta_{x_i}$ such that 
	\begin{align}\label{eq:match moments}
		\int_\X \varphi_i(x) \dd\mu^{Q_n}(x)
		= \int_\X \varphi_i(x) \dd\mu(x)\quad
		\text{for every }  i=1,\ldots,n.
	\end{align}
	A simple way to {\it approximately}
	construct this $\mu^{Q_n}$ is to first, sample $N \gg n$ points, $(y_i)_{i=1}^N$, from $\mu$ such that their empirical measure, $\tilde{\mu}_N=\frac{1}{N} \sum_{i=1}^N \delta_{y_i}$, is a good approximation to $\mu$ in the sense that $\int \varphi_i\dd\tilde{\mu}_N \approx \int \varphi_i \dd\mu$ for $i=1,\ldots,n-1$, and secondly, apply a so-called recombination algorithm
	(Remark \ref{rem:mar})
	that takes as input $(y_i)_{i=1}^N$ and $n$ functions $\varphi_1,\ldots,\varphi_{n-1}$ and outputs a measure $\mu^{Q_n}=\sum w_i \delta_{x_i}$ by selecting a subset $(x_i)_{i=1}^n$ of the points $(y_i)_{i=1}^N$ and giving them weights $(w_i)_{i=1}^n$ such that $\mu^{Q_n}$ is a probability measure that satisfies
	the equation
	\eqref{eq:match moments}
	with $\mu$ replaced by $\tilde{\mu}_N$.
	
	The challenging parts of this approach are 
	\begin{enumerate*}[label=(\roman*)]
		\item\label{itm:find functions} to construct functions $\varphi_1,\ldots,\varphi_{n-1}$ that approximately span the RKHS $\Hil_k$ for a small $n$;
		\item\label{itm:bounds} to arrive at good quantitative bounds despite the (probabilistic) sampling error resulting from the use of the empirical measure $\tilde{\mu}_N$, and the function approximation error via $\varphi_1,\ldots,\varphi_{n-1}$.
	\end{enumerate*}
	To address \ref{itm:find functions} we look for functions such that
	\begin{align}\label{eq:finite approximation}
		k(x,y) \approx k_0(x,y) \coloneqq \sum_{i=1}^{n-1}
		c_i\varphi_i(x)\varphi_i(y)
	\end{align}
	with some $c_i\ge0$.
	Two classic ways to do this are the Mercer and Nystr\"om approximations. 
	The remaining, item \ref{itm:bounds} is our main contribution.
	Theorem \ref{thm:main} shows that the worst-case error, \eqref{eq:wce}, is controlled by the sum of two terms: the first term stems from the kernel approximation  \eqref{eq:finite approximation}, the second term stems from the sample error.

	\begin{thm}\label{thm:main}
		Let $\mu$ be a Borel probability measure on $\X$ and $k$ a positive semi-definite  kernel on $\X$ such that $\int_\X k(x,x) \dd\mu(x) < \infty$. 
		Further, let $n$ be a positive integer and assume
		$k_0$ is a positive semi-definite  kernel on $\X$ such that
		\[
		\text{1. $k-k_0$ is a positive semi-definite  kernel on $\X$},
		\quad \text{and} \quad
		\text{2. $\dim \Hil_{k_0} < n$.}
		\]
		There exists a function $\operatorname{KQuad}$ such that if $D_N$ is a set of $N$ i.i.d.~samples from $\mu$, then $Q_n=\operatorname{KQuad}(D_N)$ is a random $n$-point convex quadrature that satisfies
		\begin{align}\label{eq:wce}
			\mathbb{E}_{D_N}
			\bigl[\wce(Q_n; \Hil_k, \mu)^2\bigr]
			\le 8\int_\X (k(x, x)-k_0(x,x))\dd\mu(x) + \frac{2c_{k,\mu}}N.
		\end{align}
		where $c_{k,\mu}\coloneqq \int_\X k(x, x)\dd\mu(x)- \iint_{\X\times\X} k(x, y) \dd\mu(x)\dd\mu(y)$.
		
		Moreover, the support of $Q_n$ is a subset of $D_N$ and given
		functions $\phi_1,\ldots,\phi_{n-1}\in L^1(\mu)$
		with $\Hil_{k_0}\subset\mathop\mathrm{span}\{
		\phi_1,\ldots,\phi_{n-1}\}$, $Q_n=\operatorname{KQuad}(D_N)$ can be computed with Algorithm
		\ref{algo:main} in $\ord{nN + n^3\log(N/n)}$
		computational steps.
	\end{thm}
	The function $\operatorname{KQuad}$ is deterministic but since $D_N$ is random, the resulting quadrature rule $Q_n$ is random, hence also the resulting worst case error $\wce(Q;\Hil_k, \mu)$ and the expectation in \eqref{eq:wce} denotes the expectation over the $N$ samples in $D_N$.
	The theoretical part of Theorem \ref{thm:main}
	follows from more general results that we present and prove in the Appendix: Theorem \ref{thm:emp-unify} proves the inequality,
	essentially by comparing $\Hil_k$ with $\Hil_{k_0}$; Theorem \ref{thm:uni-exist} proves the existence.
	The algorithmic part of Theorem\ref{thm:main} is discussed in Section \ref{sec:algo} below.
	Theorem \ref{thm:main} covers our two main examples for the construction of $k_0$, resp.~the choice of $\varphi_1,\ldots,\varphi_{n-1}$, and for which the error estimate gets quite explicit: the Mercer approximation, see Section \ref{sec:mer}, and the Nystr\"om approximation, see Section \ref{sec:nys}. 
	The former requires some knowledge about the spectrum of the kernel which is, however, known for many popular kernels; the latter works in full generality but yields worse theoretical guarantees for the convergence rate. 
	Finally, we emphasize that $N$ and $n$ in Theorem~\ref{thm:main} can be chosen independently and we will see that from a computational point the choice $N \sim n^2$ is preferable in which case \eqref{eq:wce}
	is faster rate than Monte Carlo,
	see also Table \ref{table:comp}.
	
	\subsection{Algorithm}\label{sec:algo}
	
	\begin{algorithm}[H]
		\caption{Kernel Quadrature with Convex Weights via Recombination $\operatorname{KQuad}$}\label{algo:main}
		\begin{algorithmic}[1]
			\Require{A positive semi-definite kernel $k$ on $\X$, a probability measure $\mu$ on $\X$, integers $N \ge n \ge1$,
				another kernel $k_0$, functions $\phi_1,\ldots,\phi_{n-1}$
				on $\X$     with $\Hil_{k_0}\subset \mathop\mathrm{span}\{\phi_1,\ldots,\phi_{n-1}\}$} and a set $D_N$ of $N$ i.i.d.\,samples from $\mu$.
			\Ensure{A set $Q_n  \coloneqq  \{(w_i,x_i)\mid i=1,\ldots,n\}\subset \R \times \X$ with $w_i\ge0$, $\sum_{i=1}^n w_i =1$}  
			\State{Apply a Recombination Algorithm (Remark \ref{rem:mar}) with
				$\bm\psi = (\phi_1, \ldots, \phi_{n-1},
				k_{1,\mathrm{diag}})^\top$,
				to the empirical measure
				$\frac1N\sum_{y \in D_N} \delta_{y}$ to
				obtain points
				$\{\tilde{x}_1, \ldots, \tilde{x}_{n+1}\}
				\subset D_N$ and weights
				$\bm{v}=
				(v_1, \ldots, v_{n+1})^\top\ge\bm{0}$
				that satisfy $\bm1^\top\bm{v} = 1$
				and
				$\bm\psi(\tilde{\bm{x}})\bm{v}
				=\frac1N\sum_{i=1}^N\bm\psi(\tilde{x}_i)$,
				where
				$\bm\psi(\tilde{\bm{x}})=[\bm\psi
				(\tilde{x}_1),
				\ldots, \bm\psi(\tilde{x}_{n+1})]\in \R^{n\times(n+1)}$.
			}\label{state:recombination}
			\State{Apply SVD with the matrix 
				$A=[\phi_{i-1}(y_j)]_{ij}\in\R^{n\times(n+1)}$ with $\phi_0=1$ to find a nonzero vector
				$\bm{u}\in\R^{n+1}$ such that
				$A\bm{u} = \bm0$ and
				$k_{1,\mathrm{diag}}
				(\tilde{\bm{x}})^\top\bm{u} \ge0$}
			\State{Compute the smallest $\alpha\ge0$
				such that $\bm{v} - \alpha\bm{u} \ge\bm{0}$
				and $v_j - \alpha u_j =0$ for some $j$}
			\State{Return $(w_i)_{i=1}^n \leftarrow
				(v_k-\alpha u_k)_{k\in I}$
				and $(x_i)_{i=1}^n \leftarrow
				(\tilde{x}_k)_{k\in I}$,
				where $I = \{1, \ldots, n+1\}\setminus\{j\}$}
		\end{algorithmic}
	\end{algorithm}
	
	Suppose we are given
	$k_0$ and $\phi_1, \ldots, \phi_{n-1}\in L^1(\mu)$
	with $\Hil_{k_0}\subset\mathop\mathrm{span}\{
	\phi_1,\ldots,\phi_{n-1}\}$,
	and also
	$N$ independent samples from $\mu$ denoted by
	$D_N = (y_1, \ldots, y_N)$.
	Theorem \ref{thm:emp-unify} in the Appendix shows that 
	if we construct a convex quadrature
	$Q_n=(w_i, x_i)_{i=1}^n$ satisfying
	\begin{equation}
		\sum_{i=1}^nw_i\bm\phi(x_i)
		=\frac1N\sum_{i=1}^N\bm\phi(y_i),
		\qquad
		\sum_{i=1}^nw_ik_{1,\mathrm{diag}}(x_i)
		\le \frac1N\sum_{i=1}^Nk_{1,\mathrm{diag}}(y_i),
		\label{eq:nys-constraints}
	\end{equation}
	where $\bm\phi =(\phi_1,\ldots,\phi_{n-1})^\top$
	and $k_{1,\mathrm{diag}}(x) = k(x, x) - k_0(x, x)$,
	it satisfies the bound \eqref{eq:wce}.
	For this problem,
	we can use the so-called {\it recombination}
	algorithms:
	\begin{remark}[Recombination]\label{rem:mar}
		Given $d-1$ functions (called test functions)
		and a probability measure supported on $N>d$ points, there exists a probability measure supported on a subset of $d$ points that gives the same mean to these $d-1$ functions.
		This follows from Carath\'eodory's theorem and is known as recombination.  Efficient deterministic \citep{lit12, maa19, tch15} as well as randomized \citep{cos20} algorithms exist to compute the new probability measure supported on $d$ points; e.g.~deterministic algorithms perform the recombination, step~\ref{state:recombination}, in $\ord{c_\phi N + d^3\log(N/d)}$ time,
		where $c_\phi$ is the cost of computing
		all the test functions at one sample.
		If each function evaluation is in constant time,
		$c_\phi = \ord{d}$.
	\end{remark}
	Let us briefly provide the intuition behind the deterministic recombination algorithms.
	We can solve the problem of ``reducing (weighted) $2d$ points to $d$ points
	in $\R^d$ while keeping the barycenter''
	by using linear programming or a variant of it.
	If we apply this to $2d$ points each given by a barycenter of approximately $\frac{N}{2d}$ points,
	we can reduce the original problem of size $N$
	to a problem of size $d\cdot\frac{N}{2d} = \frac{N}2$.
	By repeating this procedure $\log_2(\frac{N}{d})$ times we obtain the desired measure.
	
	Although the recombination introduced here only treats
	the equality constraints in \eqref{eq:nys-constraints}
	we can satisfy the remaining constraints just with $n$
	points by modifying it.
	This is done in Algorithm \ref{algo:main} which works as follows:
	First, via recombination,
	find an $(n+1)$-point convex quadrature
	$R_{n+1} = (v_i, y_i)_{i=1}^{n+1}$
	that exactly integrates functions
	$\phi_1, \ldots, \phi_{n-1}, k_{1,\mathrm{diag}}$
	with regard to the empirical measure
	$\frac1N\sum_{i=1}^N\delta_{y_i}$.
	Second, to reduce one point,
	find a direction ($-\bm{u}$ in the algorithm)
	in the space of weights on $(\tilde{x}_i)_{i=1}^{n+1}$
	that does not change the integrals of
	$\phi_1, \ldots, \phi_{n-1}$ and
	the constant function $1$,
	and does not increase the integral of
	$k_{1,\mathrm{diag}}$.
	Finally, move the weight from $\bm{v}$
	to the above direction until an entry becomes zero,
	at $\bm{v} - \alpha\bm{u}$.
	Such an $\alpha\ge0$ exists,
	as $\bm{u}$ must have a positive entry since
	it is a nonzero vector whose entries sum up to one.
	Now we have a convex weight vector with
	at most $n$ nonzero entries, so
	it outputs the desired quadrature
	satisfying \eqref{eq:nys-constraints}.
	
	\subsection{Mercer Approximation}\label{sec:mer}
	In this section and Section \ref{sec:nys},
	we assume that
	$k$ has a pointwise convergent Mercer decomposition
	$k(x, y)
	=\sum_{m=1}^\infty
	\sigma_me_m(x)e_m(y)$
	with
	$\sigma_1\ge\sigma_2\ge\cdots\ge0$
	and $(e_m)_{m=1}^\infty\subset L^2(\mu)$
	being orthonormal \citep{ste12}.
	If we let $\K$
	be the integral operator $L^2(\mu)\to L^2(\mu)$
	given by $f\mapsto \int_\X k(\cdot, y)f(y)\dd\mu(y)$,
	then $(\sigma_m, e_m)_{m=1}^\infty$ are the eigenpairs of this operator.

	The first choice
	of the approximate kernel $k_0$
	is just the trucation of
	Mercer decomposition.
	\begin{cor}\label{main-bdd}
		Theorem \ref{thm:main} applied with $k_0(x,y) =
		\sum_{m=1}^{n-1}\sigma_m e_m(x)e_m(y)$ yields a random convex quadrature rule $Q_n$ such that
		\begin{equation}
			\mathbb{E}_{D_N}\bigl[\wce(Q_n; \Hil_k, \mu)^2\bigr]
			\le 8
			\sum_{m=n}^\infty \sigma_m
			+\frac{2c_{k,\mu}}N.
			\label{eq:main-bdd}
		\end{equation}
	\end{cor}
	\begin{proof}
		It suffices to prove the result under the assumption
		$\int_\X k(x, x)\dd\mu(x)=
		\sum_{m=1}^\infty\sigma_m <\infty$,
		as otherwise the right-hand side of \eqref{eq:main-bdd}
		is infinity.
		
		For $k_1\coloneqq k - k_0$,
		we have that
		$k_1(x, y) = \sum_{m=n}^\infty\sigma_me_m(x)e_m(y)$
		and it is the inner product of
		$\Phi(x) \coloneqq (\sqrt{\sigma_m}e_m(x))_{m=n}^\infty$
		and $\Phi(y)$ in $\ell^2(\{n, n+1,\ldots\})$
		and so positive semi-definite.
		Thus $k$ and $k_0$ satisfies the assumption of
		Theorem \ref{thm:main},
		and $\int_\X k_1(x, x)\dd\mu(x) = \sum_{m=n}^\infty\sigma_m$
		applied to \eqref{eq:wce} yields the desired inequality. 
	\end{proof}
	
	\subsection{Nystr{\"o}m Approximation}\label{sec:nys}
	Although the Nystr{\"o}m method \citep{wil00,dri05,kum12} is primarily used for approximating a large Gram matrix by a low rank matrix, it can also be used for directly approximating the kernel function itself and this is how we use it.
	Given a set of $\ell$ points $Z = (z_i)_{i=1}^\ell\subset\X$, the vanilla Nystr{\"o}m approximation of $k(x, y)$ is given by
	\begin{equation}
		k(x, y) = \ip{k(\cdot, x), k(\cdot, y)}_{\Hil_k}
		\approx \ip{P_Z k(\cdot, x), P_Z k(\cdot, y)}_{\Hil_k} =: k^Z(x, y),
		\label{nys-v-int}
	\end{equation}
	where $P_Z :\Hil_k \to \Hil_k$ is a projection operator onto
	$\mathop\mathrm{span}\{k(\cdot, z_i)\}_{i=1}^\ell$.
	In matrix notation, we have
	\begin{equation}
		k^Z(x, y) = k(x, Z)W^+ k(Z, y)  \coloneqq  [k(x, z_1), \ldots, k(x, z_\ell)]
		W^+
		\left[\begin{array}{c}k(z_1, y)\\ \vdots \\ k(z_\ell, y)\end{array}\right],
		\label{nys-v-mat}
	\end{equation}
	where $W = (k(z_i, z_j))_{i, j = 1}^\ell$ is the Gram matrix for $Z$
	and $W^+$ denotes its Moore--Penrose inverse.
	We discuss the equivalence between \eqref{nys-v-int} 
	and \eqref{nys-v-mat} in Appendix \ref{sec:equiv-nys}.
	As $k^Z$ is an $\ell$-dimensional kernel, there exists an $(\ell + 1)$-point
	quadrature formula that exactly integrates functions in $\Hil_{k^Z}$.
	For a quadrature formula, exactly integrating all the functions in $\Hil_{k^Z}$
	is indeed equivalent to exactly integrating $k(z_i, \cdot)$ for all $1\le i\le \ell$,
	as long as the Gram matrix $k(Z, Z)$ is nonsingular.
	Proposition \ref{prop4} in the Appendix provides bound for the associated worst case error. 
	From this viewpoint, the Nystr{\"o}m approximation offers a natural set of test functions.
	
	The Nystr{\"o}m method has a further generalization with a low-rank approximation
	of $k(Z, Z)$.
	Concretely, by letting $W_s$ be the best rank-$s$ approximation of $W = k(Z, Z)$
	(given by eigendecomposition),
	we define the following $s$-dimensional kernel:
	\begin{equation}
		k^Z_s(x, y) \coloneqq k(x, Z)W_s^+ k(Z, y).
		\label{nys-s}
	\end{equation}
	Let $W = U\Lambda U^\top$ be the eigendecomposition
	of $W$, where
	$U = [u_1, \ldots, u_\ell] \in \R^{\ell\times\ell}$
	is a real orthogonal matrix and
	$\Lambda = \mathop\mathrm{diag}(\lambda_1, \ldots, \lambda_\ell)$
	with $\lambda_1 \ge \cdots \ge \lambda_\ell \ge 0$.
	Then, if $\lambda_s > 0$ we have
	\begin{equation}
		k^Z_s(x, y) = \sum_{i = 1}^s \frac1{\lambda_i}
		(u_i^\top k(Z, x))(u_i^\top k(Z, y)).
		\label{eq:nys-sum}
	\end{equation}
	So we can use functions $u_i^\top k(Z, \cdot)$ ($i = 1, \ldots, s$)
	as test functions,
	which is chosen from a larger dimensional space
	$\mathop\mathrm{span}\{k(z_i, \cdot)\}_{i=1}^\ell$.
	Although
	closer to the original usage of the Nyst{\"o}m method is
	to obtain $u_i^\top k(Z, \cdot)$ as an approximation
	of $i$-th eigenfunction of the integral operator $\K$
	with $Z$ appropriately chosen with respect to $\mu$,
	we have adopted an explanation suitable for
	the machine learning literature \citep{dri05,kum12}.
	
	The following is a continuous analogue of
	\citet[Theorem 2]{kum12}
	showing the effectiveness of the Nystr{\"o}m method.
	See also \citet{jin13} for an analysis specific to the case $s = \ell$.
	\begin{thm}\label{thm:nys-op-norm}
		Let $s\le \ell$ be positive integers and $\delta>0$.
		Let $Z$ be an $\ell$-point independent sample from $\mu$.
		If we define the integral operator
		$\mathcal{K}^{Z}_s : L^2(\mu) \to L^2(\mu)$
		by $f\mapsto \int_\X k^Z_s(\cdot, y)f(y)\dd\mu(y)$,
		then we have, with probability at least $1-\delta$,
		in terms of the operator norm,
		\begin{equation}
			\lVert\mathcal{K}^{Z}_s - \mathcal{K}\rVert
			\le \sigma_{s+1} +
			\frac{2\sup_{x\in\X} k(x, x)}{\sqrt{\ell}}
			\left(1 + \sqrt{2\log \frac1\delta}\right).
			\label{wtp1}
		\end{equation}
	\end{thm}
	The proof is given in Appendix \ref{sec:proof-nys}.
	By using this estimate,
	we obtain the following guarantee
	for the random convex quadrature given
	by Algorithm \ref{algo:main}
	and the Nystr{\"o}m approximation.
	\begin{cor}\label{cor:nys}
		Let $D_N$ be $N$-point independent sample from $\mu$ and let $Z$ be an $\ell$-point independent sample from $\mu$.
		Theorem \ref{thm:main} applied with the Nystr\"om approximation
		$k_0 = k_{n-1}^Z$
		yields an random $n$-point convex quadrature rule
		$Q_n$ such that, with probability at least $1-\delta$
		and 
		$k_{\max}:=\sup_{x\in\X}k(x, x)$,
		\begin{align*}
			\mathbb{E}_{D_N}\bigl[\wce
			(Q_n; \Hil_k, \mu)^2
			\,\big\vert\, Z \bigr] &\le
			8\Biggl(n\sigma_n + \sum_{m>n}\sigma_m\Biggr)
			+
			\frac{16(n-1)k_{\max}}{\sqrt{\ell}}
			\Biggl(1 + \sqrt{2\log \frac1\delta}\Biggr)
			+\frac{2c_{k, \mu}}{N}.
		\end{align*}
	\end{cor}
	\begin{proof}
		From \eqref{eq:nys-sum},
		$k^Z(x, y) - k_{n-1}^Z(x, y)
		= \sum_{i=n}^\ell \lambda_i^{-1}(u_i^\top k(Z, x))
		(u_i^\top k(Z, y))$ (ignore the terms with $\lambda_i=0$
		if necessary),
		and it is thus positive semi-definite.
		If we define $P_Z^\perp:\Hil_k\to\Hil_k$ as
		the projection operator onto the orthogonal complement of
		$\mathop\mathrm{span}\{k(\cdot, z_i)\}_{i=1}^\ell$,
		then, from \eqref{nys-v-int}, we also have
		$k(x, y) - k^Z(x, y) = \ip{P_Z^\perp k(\cdot, x),
			P_Z^\perp k(\cdot, y)}_{\Hil_k}$,
		so $k-k^Z$ is also positive semi-definite.
		In particular, $k - k_{n-1}^Z = (k- k^Z) + (k^Z-k_{n-1}^Z)$
		is positive semi-definite.
		Also, it suffices to prove the result
		when $\sum_{m=1}^\infty\sigma_m<\infty$,
		so we can now apply Theorem \ref{thm:main}.
		
		For $k_1\coloneqq k-k_{n-1}^Z$,
		we prove the inequality
		$
		\int_\X k_1(x, x)\dd\mu(x)
		= \sum_{m=1}^\infty
		\ip{e_m, (\K-\K_s^Z)e_m}_{L^2}
		\le (n-1)\lVert\K - \K_s^Z\rVert
		+ \sum_{m\ge n}\sigma_m$ 
		(see \eqref{eq:nys-estimate} in Appendix \ref{sec:app-nys}
		for details),
		and the desired inequality
		follows by combining Theorem \ref{thm:main} and Theorem
		\ref{thm:nys-op-norm} (i.e., \eqref{eq:wce} and \eqref{wtp1}).
	\end{proof}
	
	\begin{remark}\label{rem:comp-nys}
		Algorithm~\ref{algo:main} with the Nyst{\"o}m approximation
		can be decomposed into two parts: (a) Nyst{\"o}m approximation
		by truncated singular value decomposition (SVD) (the first $n-1$
		eigenvectors from an $\ell$-point sample),
		(b) Recombination from an $N$-point empirical measure.
		The complexity of (a) is $\ord{n\ell^2}$,
		and it can also be approximated by randomized SVD in $\ord{n^2\ell + \ell^2\log n}$
		\citep{hal11}.
		The cost of part (b) is
		$\ord{n\ell N + n^3\log(N/n)}$,
		where $n\ell N$ stems from the evaluation of $k_{1, \mathrm{diag}}$
		for all $N$ sampling points.
		If we do not impose the inequality constraint regarding $k_{1, \mathrm{diag}}$,
		which still works well in practice,
		the cost of part (b)
		becomes $\ord{\ell N + n^2\ell\log(N/n)}$, by using the trick
		$\frac1N\sum_{i=1}^N U_{n-1}^\top k(Z, y_i) =U_{n-1}^\top \frac1N\sum_{i=1}^N k(Z, y_i)$,
		where $U_{n-1} = [u_1, \ldots, u_{n-1}]\in\R^{{\ell\times(n-1)}}$
		is a truncation of the matrix that appears in the Nystr{\"o}m approxiamtion (\ref{nys-s},\ref{eq:nys-sum}).
		So the overall complexity is $\ord{n\ell N +n\ell^2 + n^3\log(N/n)}$
		while an approximate algorithm (randomized SVD, without the inequality constraint)
		runs in
		$\ord{\ell N + \ell^2\log n + n^2\ell\log(N/n)}$.
	\end{remark}
	
	\subsection{Kernel Quadrature Using Expectations of Test Functions}\label{sec:main-rchq}
	Algorithm \ref{algo:main}
	and the bound \eqref{eq:wce} can be generally applicable once we obtain
	a low-rank approximation $k_0$
	as we have seen in Section \ref{sec:mer} and \ref{sec:nys}.
	However, since by construction we start by reducing the empirical measure given
	by $D_N$, it is inevitable to have the $\Omega(1/N)$ term in the error estimate
	and performance.
	We can avoid this limitation by exploiting additional knowledge of expectations.

	Let $k_0$ and $k_1$ be positive definite kernels with $k = k_0 + k_1$.
	Let $\bm\phi = (\phi_1, \ldots, \phi_{n-1})^\top$ be the vector
	of test functions that spans $\H_{k_0}$.
	When we know the expectations of them, i.e., $\int_\X\bm\phi(x)\dd\mu(x)$,
	we can actually construct a convex quadrature $Q_n=(w_i, x_i)_{i=1}^n$
	satisfying
	\begin{equation}
	    \sum_{i=1}^n w_i\bm\phi(x_i) = \int_\X \bm\phi(x)\dd\mu(x),
	    \qquad
	    \sum_{i=1}^nw_ik_1(x_i, x_i)
		\le \int_\X k_1(x, x)\dd\mu(x)
		\label{cond:rchq}
	\end{equation}
	with a positive probability
	by an algorithm based on random convex hulls (Appendix~\ref{app:known expectations}, Algorithm~\ref{algo:rchq}).
	
	For this $Q_n$, we have the following theoretical guarantee (see Theorem \ref{thm:unify} in Appendix~\ref{sec:app-theory}):
	\begin{thm}
	    If a convex quadrature $Q_n$ satisfies the condition \eqref{cond:rchq}, then we have
	    \[
	        \wce(Q_n; \Hil_k, \mu)^2 \le 4\int_\X k_1(x, x)\dd\mu(x).
	    \]
	\end{thm}
	If $k_0$ is given the Mercer/Nystr{\"o}m approximations,
	we immediately have the following guarantees;
	they correspond to {\bf Mercer} and {\bf Nystr{\"o}m} in Table~\ref{table:comp}.
	See also Theorem \ref{thm:app-mer} and \ref{thm:nys-exact} for details.
    \begin{itemize}
        \item 
            If $k_0(x, y) = \sum_{m=1}^{n-1}\sigma_me_m(x)e_m(y)$ is given by the Mercer approximation,
    	    we have
    	    \[
    	        \wce(Q_n; \Hil_k, \mu)^2 \le 4\sum_{m=n}^\infty \sigma_m
    	    \]
    	    for a convex quadrature $Q_n$ satisfying \eqref{cond:rchq}.
    	\item
    	    Let $k_0 = k_{n-1}^Z$ be given by the Nystr{\"o}m approximation \eqref{nys-s}
    	    with $Z$ being an $\ell$-point independent sample from $\mu$ (with $\ell > n$).
    	    Then, for a convex quadrature $Q_n$ satisfying \eqref{cond:rchq},
    	    with probability at least $1-\delta$ (with respect to $Z$)
    		and 
    		$k_{\max}:=\sup_{x\in\X}k(x, x)$,
    		we have
    		\begin{align*}
    			\wce
    			(Q_n; \Hil_k, \mu)^2
    			\le
    			4\Biggl(n\sigma_n + \sum_{m>n}\sigma_m\Biggr)
    			+
    			\frac{8(n-1)k_{\max}}{\sqrt{\ell}}
    			\Biggl(1 + \sqrt{2\log \frac1\delta}\Biggr).
    		\end{align*}
    \end{itemize}
	
	\section{Numerical Experiments}\label{sec:num}
	In this section, we compare our methods with several existing methods.
	In all the experiments, we used the setting where we can compute
	$\int_\X k(x, y)\dd\mu(y)$
	for $x\in \X$ and $\iint_{\X\times\X} k(x, y)\dd\mu(x)\dd\mu(y)$
	since then we can evaluate the worst-case error
	of quadrature formulas explicitly.
	Indeed, if a quadrature formula $Q_n$ is
	given by points $X = (x_i)_{i=1}^n$
	and weights $\bm{w} = (w_i)_{i=1}^n$, then we have
	\begin{align}
		\wce(Q_n; \Hil_k, \mu)^2
		= 
		\bm{w}^\top k(X, X) \bm{w}
		-2 \mathbb{E}_{y}[\bm{w}^\top k(X, y)]
		+ \mathbb{E}_{y,y^\prime}[k(y, y^\prime)]
		\label{eq:wce-general}
	\end{align}
	for independent $y,y^\prime\sim\mu$
	under $\int_\X \sqrt{k(x, x)}\dd\mu(x) < \infty$,
	which is a well-known formula for the worst-case error
	\citep{gre06,sri10}.
	An essential remark shown in \citet{hus12}
	is that the Bayesian quadrature \citep{oha91}
	with covariance kernel $k$ given observation at points $(x_i)_{i=1}^n$
	(automatically) estimates the integral as
	$\sum_{i=1}^nw_if(x_i)$ with $(w_i)_{i=1}^n$ minimizing the above expression.
	Once given points $(x_i)_{i=1}^n$ and additional knowledge of expectations,
	we can compute the optimal weights $(w_i)_{i=1}^n$
	by solving a convex quadratic programming (CQP),
	either without any restrictions or with the condition that
	$(w_i)_{i=1}^n$ is convex.
	Although the former can be solved by matrix inversion,
	we have used the optimizer Gurobi\footnote{Version 9.1.2, \url{https://www.gurobi.com/}} for both CQPs to avoid numerical instability.
	For the recombination part,
	we have modified the Python library by \citet{cos20} implementing the algorithm
	of \citep{tch15}.
	
	Our theoretical bounds are close to optimal in classic examples and we see that the algorithm even outperforms the theory in practice especially in Section \ref{sec:sob}.
	We also execute a measure reduction of a large discrete measure
	in terms of Gaussian RKHS and our methods shows a fast convergence rate
	in two ML datasets in Section \ref{sec:ml}.
	\footnote{All done on a MacBook Pro, CPU: 2.4~GHz Quad-Core Intel Core i5, RAM: 8~GB 2133~MHz LPDDR3.}
	
	\begin{figure}
		\centering
		\begin{subfigure}[h]{0.49\hsize}
			\centering
			\includegraphics[width=\hsize]{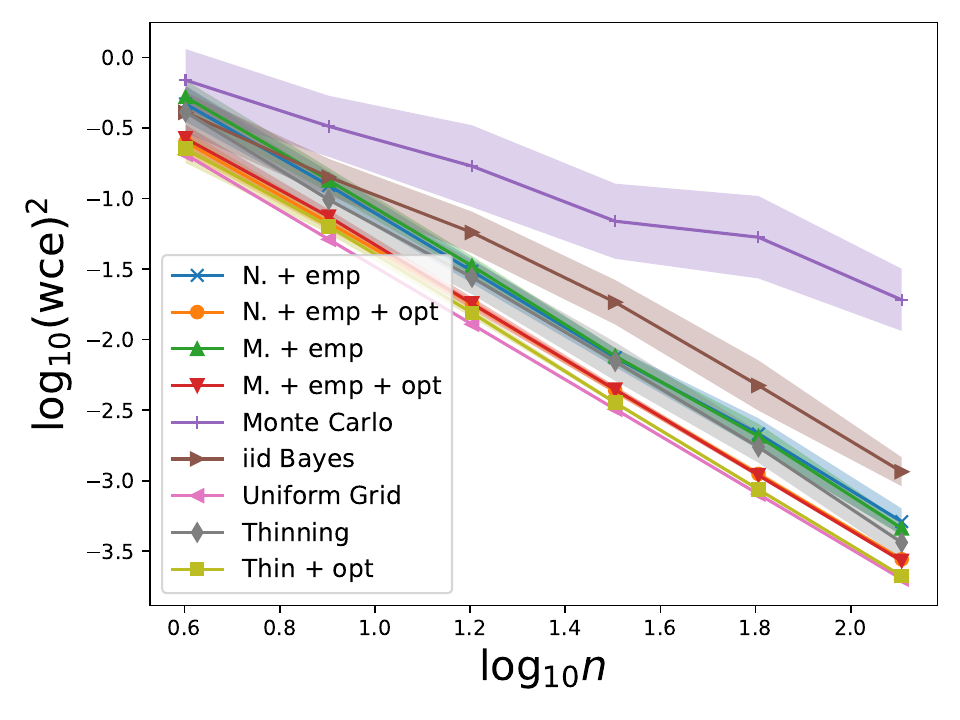}
			\caption{$d=1$, $r = 1$}
		\end{subfigure}
		\begin{subfigure}[h]{0.49\hsize}
			\centering
			\includegraphics[width=\hsize]{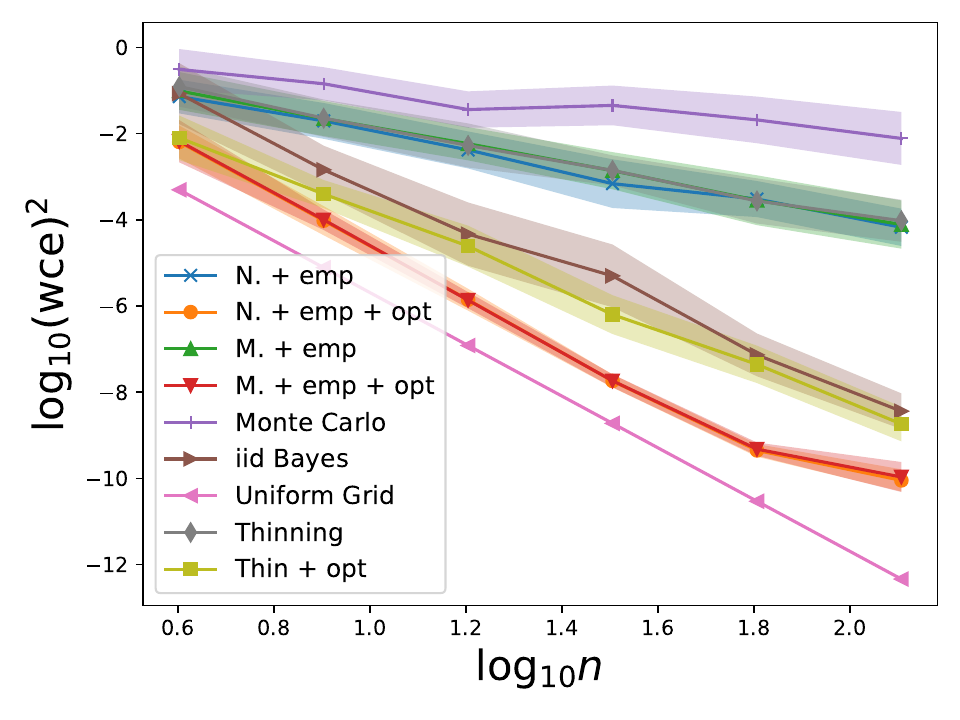}
			\caption{$d=1$, $r = 3$}
		\end{subfigure}
		\begin{subfigure}[h]{0.49\hsize}
			\centering
			\includegraphics[width=\hsize]{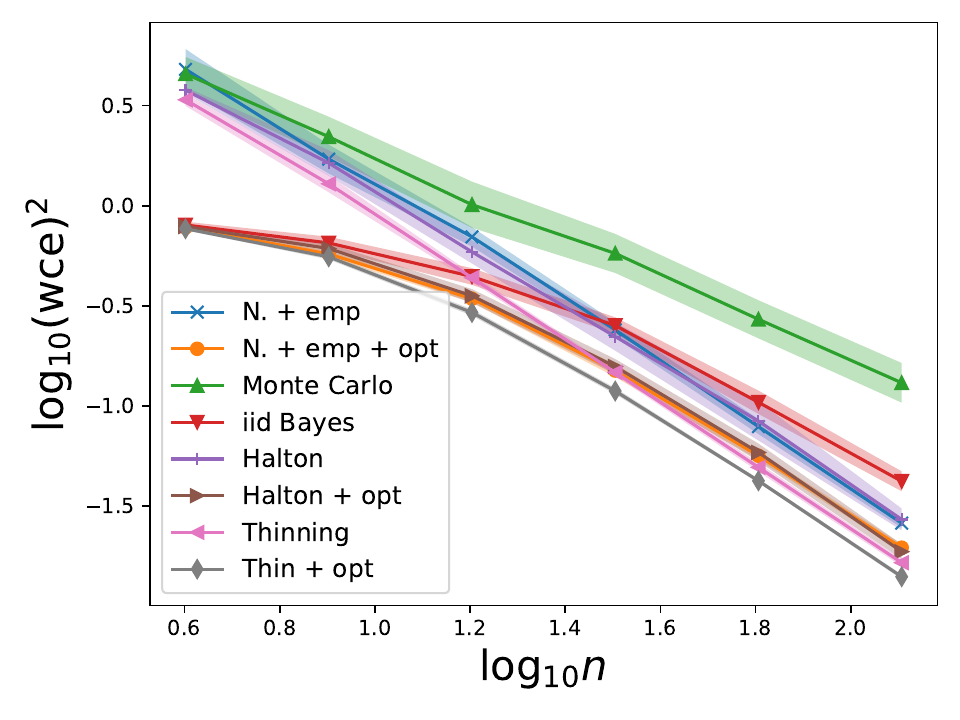}
			\caption{$d=2$, $r = 1$}
		\end{subfigure}
		\begin{subfigure}[h]{0.49\hsize}
			\centering
			\includegraphics[width=\hsize]{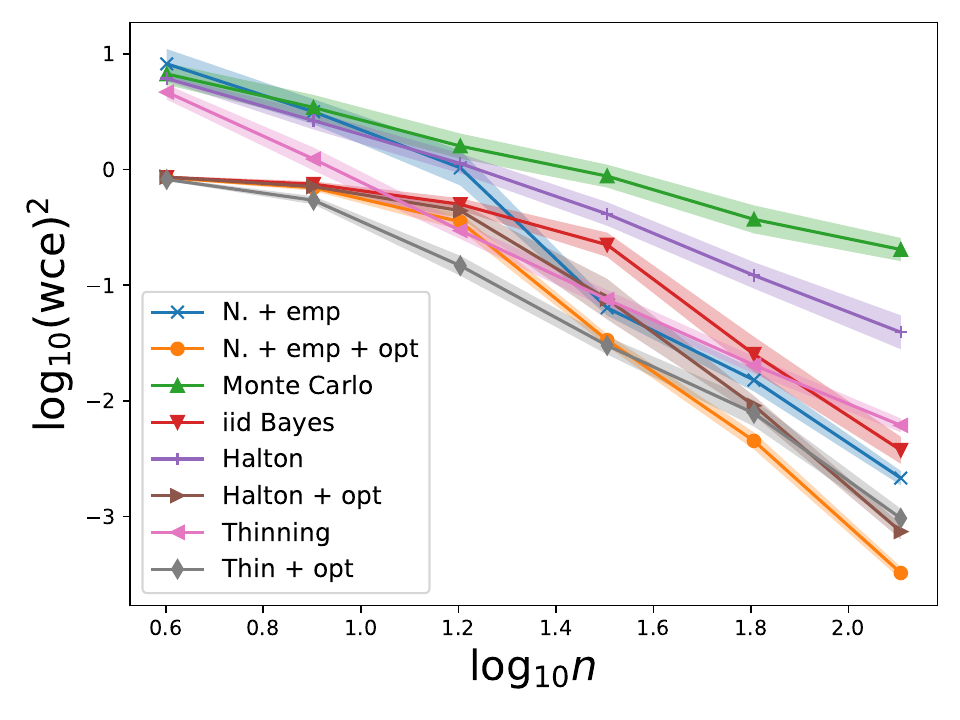}
			\caption{$d=3$, $r = 3$}
		\end{subfigure}
		\caption{Periodic Sobolev spaces with kernel $k_r^{\otimes d}$:
			The average of $\log_{10}(\wce(Q_n; \Hil_k, \mu)^2)$ over $20$ trials
			is plotted for each method of obtaining $Q_n$.
			The shaded regions show their standard deviation.
			The worst computational time per one trial
			was 57 seconds of {\bf Thin + opt}
			in $(d, r, n) = (3, 3, 128)$,
			where
			{\bf Thinning} was 56 seconds and
			{\bf N. + emp [+ opt]} was 22 seconds.}
		\label{fig-sob}
    \end{figure}
    \begin{figure}
		\centering
		\begin{subfigure}[h]{0.49\hsize}
			\centering
			\includegraphics[width=\hsize]{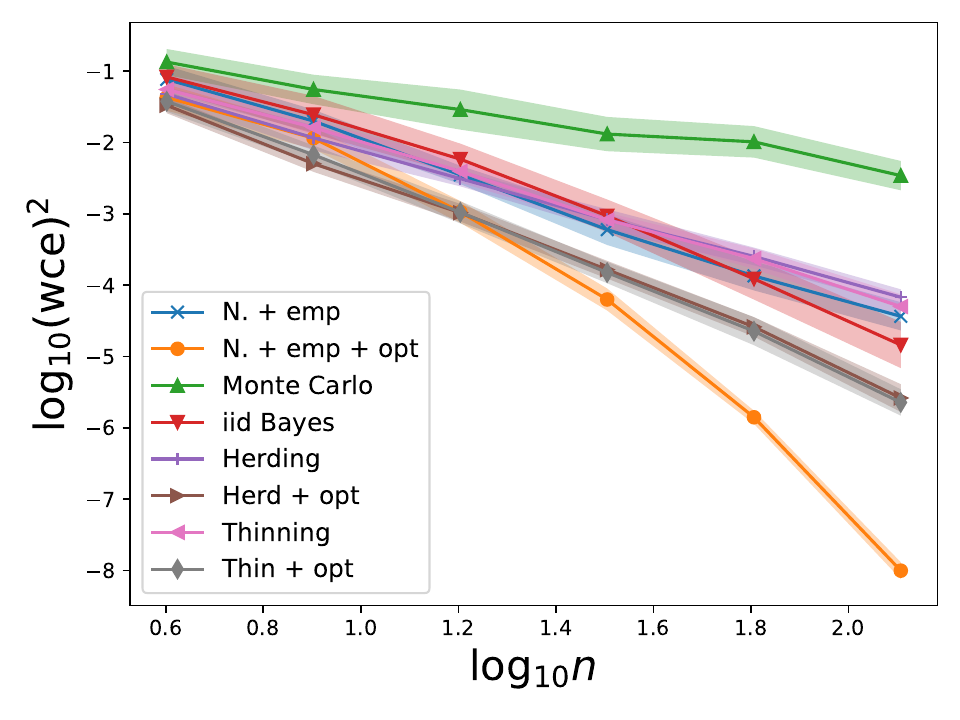}
			\caption{3D Road Network data}
		\end{subfigure}
		\begin{subfigure}[h]{0.49\hsize}
			\centering
			\includegraphics[width=\hsize]{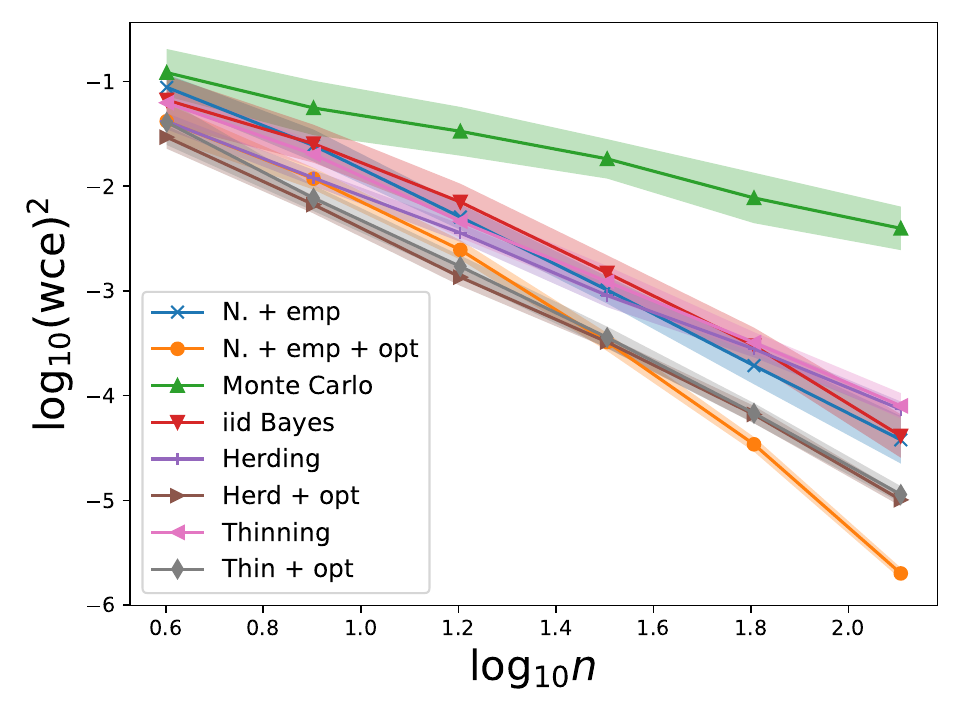}
			\caption{Power Plant data}
		\end{subfigure}
		\caption{Measure reduction in Gaussian RKHS with two ML datasets:
			The average of
			$\log_{10}(\wce(Q_n; \Hil_k, \mu)^2)$ over $20$ trials
			is plotted for each method of obtaining $Q_n$.
			The shaded regions show their standard deviation.
			The worst computational time per one trial
			was 14 seconds of {\bf Thinning [+ opt]}
			in Power Plant data with $n = 128$,
			where {\bf N. + emp [+ opt]} was 6.3 seconds.}
		\label{fig-ml}
	\end{figure}
	
	\subsection{Periodic Sobolev Spaces with Uniform Measure}\label{sec:sob}
	For a positive integer $r$,
	consider the Sobolev space of functions on $[0, 1]$
	endowed with the norm
	$\lVert f \rVert^2 = (\int_0^1 f(x)\dd x)^2
	+ (2\pi)^{2r}\int_0^1 f^{(r)}(x)^2\dd x$,
	where $f$ and its derivatives $f^{(1)}, \ldots f^{(r)}$ are periodic
	(i.e., $f(0) = f(1)$ and so forth).
	This function space can be identifies as the RKHS of the kernel
	\[
	k_r(x, y) = 1 + \frac{(-1)^{r-1}(2\pi)^{2r}}{(2r)!} B_{2r}(|x - y|)
	\]
	for $x, y\in [0, 1]$,
	where $B_{2r}$ is the $2r$-th Bernoulli polynomial \citep{wah90,bac17}.
	If we let $\mu$ be the uniform measure on $[0, 1]$,
	the normalized eigenfunctions (of the integral operator) are
	$1$, $c_m(\cdot) = \sqrt{2}\cos(2\pi m\,\cdot)$ and $s_m (\cdot) = \sqrt{2}\sin(2\pi m\,\cdot)$
	for $m = 1, 2,\ldots$,
	and the corresponding eigenvalues are $1$ and $m^{-2r}$ (both for $c_m$ and $s_m$).
	Although the rectangle formula $f\mapsto n^{-1}\sum_{i=1}^{n}f(i/n)$ (a.k.a. {\tt Uniform Grid} below)
	is known to be optimal for this kernel \citep{zhe81,nov88} in the sense of worst-case error,
	this RKHS is commonly used for testing the efficiency of general kernel quadrature methods
	\citep{bac17,bel19,kan20}.
	We also consider its multivariate extension on $[0, 1]^d$,
	i.e., the RKHS given by the product kernel $k_r^{\otimes d}(\bm{x}, \bm{y}) \coloneqq
	\prod_{i=1}^dk_r(x_i, y_i)$ for $\bm{x}=(x_1,\ldots,x_d), \bm{y} = (y_1,\ldots,y_d)
	\in [0, 1]^d$.
	
	We carried out the experiment for $(d, r)=(1, 1), (1, 3), (2, 1), (3, 3)$.
	For each $(d, r)$,
	we compared the following algorithms for
	$n$-point quadrature rules with
	$n\in \{4, 8, 16, 32, 64, 128\}$.
	\begin{description}
		\item[N. + emp, N. + emp + opt:]
		We used the functions $u_i^\top k(Z, \cdot)$ ($i = 1, \ldots, n - 1$)
		given by the Nystr{\"o}m approximation \eqref{eq:nys-sum} with $s = n-1$
		as test functions $\phi_1, \ldots, \phi_{n-1}$
		in Algorithm \ref{algo:main}.
		The set $Z$ was given as an $(\ell = )10n$-point independent sample from $\mu$.
		We used $N = n^2$ samples from $\mu$.
		In `{\bf + opt}' we additionally optimized the {\it convex}
		weights using \eqref{eq:wce-general}
		\item[M. + emp, M. + emp + opt ($d=1$):]
		We used the first $n - 1$ functions of the sequence of eigenfunctions
		$1, c_1, s_1, c_2, s_2, \ldots$
		as test functions $\phi_1, \ldots, \phi_{n-1}$ in Algorithm \ref{algo:main}.
		We used $N = n^2$ samples from $\mu$.
		In `{\bf + opt}' we additionally optimized the {\it convex}
		weights using \eqref{eq:wce-general}.
		\item[Monte Carlo, iid Bayes:]
		With an $n$-point independent sample $(x_i)_{i=1}^n$
		from $\mu$,
		we used uniform weights $1/n$ in {\bf Monte Carlo}
		and the weights optimized using \eqref{eq:wce-general} in {\bf iid Bayes}.
		\item[Uniform Grid ($d=1$):]
		We used the rectangle formula $f\mapsto n^{-1}\sum_{i=1}^nf(i/n)$.
		This is known to be optimal
		\citep[not just up to constant,
		but exactly;][]{zhe81,nov88},
		and thus equivalent to the Bayesian quadrature
		on the uniform grid,
		i.e., the weights are already optimized.
		\item[Halton, Halton + opt ($d\ge2$):]
		For an $n$-point sequence given by the Halton sequence
		with Owen scrambling \citep{hal60,owe17},
		the uniform weights $w_i=1/n$ is adopted in {\bf Halton}
		and the weights are additionally optimized using \eqref{eq:wce-general}
		in {\bf Halton + opt}.
		\item[Thinning, Thin + opt:]
		Given an $N$-point independent sample $(y_i)_{i=1}^N$ with $N=n^2$
		from $\mu$,
		an $n$-point subset $(x_i)_{i=1}^n$
		taken from a KT++ algorithm (kernel thinning \citep{dwi21,dwi22}
		combined with Compress++ algorithm \citep{she22} with the oversampling parameter
		$\mathfrak{g} = \min\{4, \log_2n\}$,
		implemented with GoodPoints package: \url{https://github.com/microsoft/goodpoints})
		is adopted in {\bf Thinning}.
		In {\bf `+ opt'} we additionally optimized the {\it convex} weights using \eqref{eq:wce-general}.
	\end{description}
	
	The results are given in Figure \ref{fig-sob}.
	In $d=1$, the optimal rate given by {\bf Uniform Grid}
	is known to be $\ord{n^{-2r}}$.
	As the uniform sampling is equal to the {\it optimized distribution}
	of \citet{bac17} in this case,
	{\bf iid Bayes} also achieves this rate up to log factors.
	Although our theoretical guarantee for {\bf M. + emp}
	is $\ord{n^{1-2r} + N^{-1}}$ with $N = n^2$ (Corollary \ref{main-bdd}),
	in the case $(d, r) = (1, 1)$,
	we can observe that in the experiment
	it is better than {\bf iid Bayes} and close to the optimal error of {\bf Uniform Grid},
	but slightly worse than {\bf Thinning}.
	Moreover, {\bf N. + emp}, which does not use
	the information of spectral decomposition,
	is remarkably almost as accurate as {\bf M. + emp} in $d=1$.
	Furthermore, if we additionally use the knowledge of expectations,
	which {\bf iid Bayes} is already doing,
	{\bf M./N. + emp + opt}
	become surprisingly accurate even with $N=n^2$.
	They are worse than {\bf Thinn + opt} when $r=1$,
	but well outperform it when $r=3$.
	Nonlineality in the graph of these methods when $(d, r, n)=(1, 3, 128)$
	should be from numerical accuracy of the CQP solver
	(see also Section \ref{sec:app-sob}).
	
	The accuracy of {\bf N. + emp + opt} becomes more remarkable
	in multivariate cases.
	It behaves almost the same as {\bf Halton + opt} in $d = 2$
	and clearly beats it in $d=3$.
	Also, the sudden jump of our methods around $n = 30$ in $(d, r) = (3, 3)$
	seems to be caused by the jump of eigenvalues.
	Indeed, for the integral operator given by $k_3^{\otimes 3}$ with uniform measure,
	the eigenspace of the largest eigenvalue $1$
	is of dimension $27$,
	and the next largest eigenvalue is $1/64$.
	Again in the latter case,
	{\bf N. + emp + opt} outperforms {\bf Thin + opt},
	and these results suggest that our method works better when there is a strong spectral decay,
	as is explicitly incorporated in our algorithm.

	Note also that we can compare Figure \ref{fig-sob}
	with \citet[][Figure 1]{bel19}
	which includes some other methods
	such as DPPs, herding and sequential Bayesian quadrature,
	as we did experiments under almost the same setting.
	In particular, in the case $(d, r)=(1, 3)$ where the eigenvalue decay is fast,
	we see that our method substantially outperforms
	the sequential Bayesian quadrature.
	
	\subsection{Measure Reduction in Machine Learning Datasets}\label{sec:ml}
	We used two datasets from UCI Machine Learning
	Repository (\url{https://archive.ics.uci.edu/ml/datasets/}).
	We set $\mu$ as the equally weighted measure over (a subset of) the data points
	$X = (X^{(1)}, \ldots, X^{(d)})^\top$ ($d = 3, 5$, respectively),
	where each entry is centered and normalized.
	We considered the Gaussian kernel $\exp(-\lVert x- y\rVert^2/(2\lambda^2))$
	whose hyperparameter $\lambda$ is determined by {\it median heuristics} \citep{gar17},
	and compared the performance of
	{\bf N. + emp}, {\bf N. + emp + opt} (with $\ell = 10n$, $N = n^2$),
	{\bf Monte Carlo}, {\bf iid Bayes}, {\bf Thinning}, {\bf Thin + opt}.
	We also added {\bf Herding},
	an equally weighted greedy algorithm with global optimization \citep{che10},
	and its weight optimization {\bf Herd + opt} within {\it convex} quadrature
	given by \eqref{eq:wce-general}.
	We conducted the experiment for $n \in \{4, 8, 16, 32, 64, 128\}$.
	
	The first is {\bf 3D Road Network Data Set} \citep{kaul2013building}.
	The original dataset is $3$-dimensional real vectors at $434874$ points.
	To be able to compute the worst-case error \eqref{eq:wce-general}
	efficiently to evaluate each kernel quadrature,
	we used a random subset $\X$
	of size $43487 = \lfloor 434874 / 10\rfloor$
	(fixed throughout the experiment)
	and defined $\mu$ as the uniform measure on it.
	We determined $\lambda$ with the median heuristic
	by using a random subset of $\X$ with size $10000$
	and used the same $\X$ and $\lambda$ throughout the experiment.
	The second is {\bf Combined Cycle Power Plant Data Set} \citep{kaya2012local,tufekci2014prediction}.
	The original dataset is $5$-dimensional real vectors at 9568 points.
	We set the whole data as $\X$
	and defined $\mu$ as the uniform measure on it.
	We determined $\lambda$ with median heuristics
	by using the whole $\X$.
	
	Figure \ref{fig-ml}
	shows the results.
	We can observe that in both experiments {\bf N. + emp + opt} successfully
	exploits the fast spectral decay of Gaussian kernel
	and significantly outperforms
	other methods.
	Also, even without using the knowledge of any expectations,
	{\bf N. + emp} (and {\bf Thinning}) show a decent convergence rate comparable to {\bf Herding}
	or {\bf iid Bayes}, which actually use the additional information.
	See also the end of Section \ref{sec:app-ml} for the plot
	of $\wce(Q_n; \Hil_k, \mu^\prime)$ for another
	set of empirical data $\mu^\prime$.
	
	\section{Concluding Remarks}\label{sec:concl}
	We leveraged a classical measure reduction tool, recombination, with spectral properties of kernels to construct kernel quadrature rules with positive weights. 
	The resulting algorithms show strong benchmark performance despite their restriction to convex weights. Our method has also recently been applied to Bayesian inference problems \citep{ada22}.

	Although our method is applicable to fairly general situations,
	the usage or performance can be limited when it is difficult or inefficient
	to directly sample from the target measure $\mu$.
	Hence, an interesting follow up questions, is how one could replace the i.i.d.~samples with smarter sampling (DPP, importance sampling, etc) before the recombination is carried out.
	Further, our theoretical results do not fully explain the empirical superiority;
	especially the $1/\sqrt{\ell}$ term does not match the experiments and it is a challenging future research question to reduce this theoretical gap.
	Nevertheless, we believe our method is the first generally applicable algorithm
	with a guarantee from the spectral decay.
	
	
	
	\begin{ack}
		The authors would like to thank Chris Oates and Toni Karvonen
		for helpful remarks and discussions.
		The authors are also grateful to anonymous reviewers for detailed and constructive discussions that improved the paper.
		Harald Oberhauser and Terry Lyons are supported by
		the DataSıg Program [EP/S026347/1], the Alan Turing Institute [EP/N510129/1], the Oxford-Man Institute, and the Hong Kong Innovation and Technology Commission (InnoHK Project CIMDA).
	\end{ack}
	
	\bibliography{cite}
	\bibliographystyle{abbrvnat}

	\newpage

\appendix
	
	\section{Outline of the Appendix}
	Appendix \ref{sec:app-theory} contains general results from which the results presented in the main text, in particular Theorem \ref{thm:main}, follow as special cases. 
	Appendix \ref{sec:proof-app} contains the proofs of these theoretical results and needed technical lemmas. 
	Appendix \ref{app:known expectations} shows that if the expectations $\int \varphi_i(x) \dd\mu(x)$ are known, then this knowledge can be used to further improve the theoretical bounds; it also gives a simple modification of Algorithm \ref{algo:main} doing this efficiently.
	Appendix \ref{app:experiments} provides additional numerical experiments and benchmarks.
	
	\section{Theoretical Results and Remarks}
	\label{sec:app-theory}
	In this section, we present theoretical results that include our main results as a special case.
	The proofs are given in Section \ref{sec:proof-app}.
	
	\paragraph{Notation.}
	For simplicity,
	for a quadrature $Q_n$ given by points $(x_i)_{i=1}^n$
	and weights $(w_i)_{i=1}^n$
	and a probability measure $\mu$,
	we denote the integration of an integrable
	function $f$ on $\X$
	with respect to these measures by
	\[
	Q_n(f) = \sum_{i=1}^n w_if(x_i),
	\qquad
	\mu(f) = \int_\X f(x)\dd\mu(x),
	\]
	respectively.
	We also write the inner-product and norm of an RKHS
	$\Hil_k$ by $\ip{\cdot, \cdot}_{\Hil_k}$
	and $\lVert\cdot\rVert_{\Hil_k}$.
	Furthermore, we use the
	probability simplex $\Delta^n$
	and convex hull $\cv A$ of a set $A\subset\R^d$
	in the proofs:
	\[
	\Delta^n:=\left\{(w_i)_{i=1}^n
	\lmid
	w_i\ge0,\, \sum_{i=1}^n w_i=1\right\},
	\ 
	\cv A \coloneqq \left\{\sum_{i=1}^n w_ia_i
	\lmid
	(w_i)\in\Delta^n,\, a_i\in A,\, n\ge1
	\right\}.
	\]
	
	\subsection{Quantitative Results}
	We work under the following setting
	as in the assumption of Theorem \ref{thm:main}.
	\begin{assp}\label{asp:ker}
		$\mu$ is a Borel probability measure on $\X$,
		and $k$ is a positive semi-definite kernel
		on $\X$ such that $\int_\X k(x,x) \dd\mu(x) < \infty$.
		Further, $k_0$ is a positive semi-definite kernel on $\X$ such that $k_1\coloneqq k-k_0$ is a positive semi-definite  kernel on $\X$.
	\end{assp}
	
	The following is a general result regarding
	a quadrature formula exactly integrating
	functions in $\Hil_{k_0}$.
	
	\begin{thm}\label{thm:unify}
		Under Assumption \ref{asp:ker},
		if an $n$-point convex quadrature $Q_n$ on $\X$ satisfies
		$Q_n(f)=\mu(f)$
		for any $f = k_0(\cdot, x)$ with $x\in\X$,
		we have
		\begin{equation}
			\wce(Q_n; \Hil_k, \mu)
			\le Q_n(g) + \mu(g),
			\label{eq:unify}
		\end{equation}
		where $g$ is the function given by
		$g(x) = \sqrt{k_1(x, x)}$.
		In particular, the following assertions hold
		for such a quadrature $Q_n$:
		\begin{itemize}
			\item[(a)]
			We have
			$\wce(Q_n; \Hil_k, \mu)\le
			2\lVert g\rVert_{\infty}=
			2\sup_{x\in\X}\sqrt{k_1(x, x)}$.
			\item[(b)]
			If we additionally have $Q_n(g)\le \mu(g)$,
			then we have
			$\wce(Q_n; \Hil_k, \mu)\le 2\mu(g)$.
			\item[(c)]
			If we additionally have $Q_n(g^2)
			\le \mu(g^2)$ instead of (b),
			we still have
			\[
			\wce(Q_n;\Hil_k, \mu)^2
			\le 4\int_\X k_1(x, x)\dd\mu(x).
			\]
		\end{itemize}
	\end{thm}
	
	\begin{remark}
		For a Borel probability measure $\nu$ on $\X$
		and a nonnegative function $h:\X\to\R_{\ge0}$,
		we have an inequality
		$\int_\X \sqrt{h(x)}\dd\nu(x)
		\le \left(\int_\X h(x)\dd\nu(x)\right)^{1/2}$,
		so the above $\mu(g)$ can be upper bounded
		by $\int_\X k_1(x, x)\dd\mu(x)$,
		which is equal to the trace of the
		integral operator given by $k_1$.
		Also, the assumption in
		Theorem \ref{thm:unify} can be weakened to
		the integrability of $\sqrt{k(x, x)}$ from the same
		inequality as you can see in the proof.
	\end{remark}
	
	We can combine Theorem \ref{thm:unify}
	with an empirical approximation of $\mu$
	to obtain the following result,
	which essentially implies Theorem \ref{thm:main}.
	\begin{thm}\label{thm:emp-unify}
		Under Assumption \ref{asp:ker},
		let $D_N$ be a set of $N$ independent samples from $\mu$,
		and $\tilde{\mu}_N$ be its  empirical measure,
		i.e., $\tilde{\mu}_N = \frac1N\sum_{y \in D_N}\delta_{y}$.
		Then, if an $n$-point convex quadrature $Q_n$ on $\X$ satisfies
		$Q_n(f)=\tilde{\mu}_N(f)$
		for any $f = k_0(\cdot, x)$ with $x\in\X$,
		we have
		\begin{equation}
			\E{\wce(Q_n; \Hil_k, \mu)^2}
			\le 2\E{(Q_n(g) + \tilde{\mu}_N(g))^2}
			+ \frac{2c_{k,\mu}}N,
			\label{eq:emp-unify}
		\end{equation}
		where
		$g(x)  \coloneqq  \sqrt{k_1(x, x)}$
		and $c_{k,\mu}  \coloneqq  \int_\X k(x, x)\dd\mu(x)
		- \iint_{\X\times\X} k(x, y) \dd\mu(x)\dd\mu(y)$.
		In particular, the following assertions hold
		for such a quadrature $Q_n$:
		\begin{itemize}
			\item[(a)]
			We have
			$\E{\wce(Q_n; \Hil_k, \mu)^2}\le
			8\sup_{x\in\X}k_1(x, x) + 2c_{k,\mu}/N$.
			\item[(b)]
			If we additionally always require
			$Q_n(g)\le \tilde{\mu}_N(g)$,
			then we have
			\[
			\E{\wce(Q_n; \Hil_k, \mu)^2}
			\le 8\int_\X k_1(x, x)\dd\mu(x) + \frac{2c_{k,\mu}}N.
			\]
			The requirement
			$Q_n(g)\le \tilde{\mu}_N(g)$
			can be replaced by
			$Q_n(g^2)\le\tilde{\mu}_{N}(g^2)$
		\end{itemize}
	\end{thm}

	Although we have assumed $k-k_0$
	is positive semi-definite in the previous assertions,
	the uniform bound works without the assumption as follows.
	\begin{prop}\label{prop4}
		Let $\mu$ be a Borel probability measure on $\X$.
		Let $k$ and $k_0$ be positive semi-definite
		kernels on $\X$
		satisfying $\int_\X \sqrt{k(x, x)}\dd\mu(x),
		\int_\X \sqrt{k_0(x, x)}\dd\mu(x)<\infty$.
		If an $n$-point convex quadrature $Q_n$ on $\X$
		satisfies $Q_n(f) = \mu(f)$
		for any $k_0(\cdot, x)$ with $x\in \X$,
		we have
		\[
		\wce(Q_n; \Hil_k, \mu)
		\le 2\sup_{x, y\in \X}
		\sqrt{\lvert k(x, y) - k_0(x, y)\rvert}.
		\]
		Furthermore, if $\dim \Hil_{k_0} < n$,
		there exists
		an $n$-point convex quadrature $Q_n$
		satisfying
		$Q_n(f)=\mu(f)$ for each $f = k_0(\cdot, x)$.
	\end{prop}
	In this paper, we focus on
	the cases where $k_0$ is either given by
	the truncated Mercer decomposition or
	Nystr{\"o}m approximation.
	For many important kernels, however,
	we may also use the random Fourier features \citep{rah07} or its periodic version \citep{tom18}
	which can easily be combined with Proposition \ref{prop4}, but it is beyond the scope of this paper to choose its appropriate variant for each kernel \citep[see][for a list of variants]{liu20}.

	\subsection{Existence Results}
	
	The existence of quadrature formulas
	satisfying the estimate of 
	Theorem \ref{thm:unify} or Theorem \ref{thm:emp-unify}
	is guaranteed when $\dim\Hil_{k_0}<n$.
	\begin{thm}\label{thm:uni-exist}
		Under Assumption \ref{asp:ker},
		if $\dim \Hil_{k_0} < n$,
		there exists
		an $n$-point convex quadrature $Q_n$
		satisfying
		$Q_n(f)=\mu(f)$ for each $f = k_0(\cdot, x)$.
		This still holds even if we additionally require
		$Q_n(g)\le\mu(g)$ or $Q_n(g^2)\le\mu(g^2)$
		for $g(x) = \sqrt{k_1(x, x)}$.
	\end{thm}
	
	\begin{remark}
		This also implies the existence result
		of $Q_n$ satisfying the condition in
		Theorem \ref{thm:emp-unify}
		if we replace $\mu$ by $\tilde{\mu}_N$.
	\end{remark}
	
	The algorithm for constructing a kernel quadrature
	with Theorem \ref{thm:emp-unify} is given
	in the main body, see Algorithm~\ref{algo:main}.
	The ones with Theorem \ref{thm:unify}
	requires further knowledge of
	the expectation of test functions,
	i.e., the values of
	$\int_\X\phi_i(x)\dd\mu(x)$
	with $\Hil_{k_0}\subset
	\mathop\mathrm{span}\{\phi_1, \ldots, \phi_{n-1}\}$.
	Under this additional information,
	we have an algorithm (Algorithm \ref{algo:rchq})
	based on random sampling
	given in the following section.
	
	
	\subsection{Eigenvalue estimate for Gaussian kernels}\label{sec:eig-gauss}
	We provide a proof of a folklore estimate on
	the eigenvalues of integral operators given by
	a Gaussian kernel.
	Let $k(x, y)=\exp(-\frac1{2\ell^2}(x - y)^2)$ for an $\ell>0$
	and $x, y\in\R$.
	Then, it has the following expansion
	\citep{min10,kar21b}:
	\begin{equation}
	    k(x, y) = \sum_{m=0}^\infty \phi_m(x)\phi_m(y),
	    \qquad
	    \phi_m(x) = \frac1{\ell^m\sqrt{m!}}x^m\exp\left(-\frac{x^2}{2\ell^2}\right).
	    \label{eq:expansion-gauss}
	\end{equation}
	Let $\mu$ be a Borel probability measure
	supported on a compact domain,
	i.e., $\mu(\{x\in\R\mid \lvert x\rvert > R\})=0$
	for some $R>0$.
	Let us consider the RKHS given by $k$ over $\X := \supp\mu$.
	
	Recall that $\sigma_n$ is the $n$-th eigenvalue of the integral operator
	\[
	    \K: L^2(\mu) \to L^2(\mu);\quad f\mapsto \K f = \int_\X k(\cdot, y)f(y) \dd\mu(y).
	\]
	From the minimax property of eigenvalues of compact Hermitian operators, we have
	\begin{align*}
	    \sigma_n 
	    &= \inf_{g_1, \ldots, g_{n-1}\in L^2(\mu)} \sup_{f\in L^2(\mu)
	    \cap \{g_1,\ldots,g_{n-1}\}^\perp,\, \lVert f\rVert_{L^2(\mu)}=1}
	    \ip{f, \K f}_{L^2(\mu)}\\
	    &\le \sup_{f\in L^2(\mu) \cap \{\phi_0,\ldots,\phi_{n-2}\}^\perp,\, \lVert f\rVert_{L^2(\mu)}=1}
	    \ip{f, \K f}_{L^2(\mu)},
	\end{align*}
	where the orgthogonal complement is taken in terms of $L^2(\mu)$-inner product
	and $\phi_m$ are functions given in \eqref{eq:expansion-gauss}.
	They are indeed in $L^2(\mu)$ as $\mu$ is compactly supported.
	
	Now, let $k_n(x, y):=\sum_{m=n-1}^\infty\phi_m(x)\phi_m(y)$.
	For an $f\in L^2(\mu)
	    \cap \{\phi_0,\ldots,\phi_{n-2}\}^\perp$,
	we have
	\begin{align*}
	    \ip{f, \K f}_{L^2(\mu)}
	    &= \iint_{\X\times\X}f(x)k(x, y)f(y)\dd\mu(y)\dd\mu(x)\\
	    &= \iint_{\X\times\X}f(x)k_n(x, y)f(y)\dd\mu(y)\dd\mu(x)\\
	    &\le \iint_{\X\times\X}f(x)\sqrt{k_n(x, x)}\sqrt{k_n(y, y)}f(y)\dd\mu(y)\dd\mu(x)
	    \tag{positive definiteness}\\
	    &=\left(\int_\X \sqrt{k_n(x, x)}f(x)\dd\mu(x)\right)^2\\
	    &\le \left(\int_\X k_n(x, x)\dd\mu(x)\right)\lVert
	    f\rVert_{L^2(\mu)}^2.
	    \tag{Cauchy--Schwarz}
	\end{align*}
	Therefore, we have the estimate
	$\sigma_n\le \int_\X k_n(x, x)\dd\mu(x)$.
	We have
	\[
	    k_n(x, x)
	    = \sum_{m=n-1}^\infty\frac1{m!}\left(\frac{x}{\ell}\right)^{2m}\exp\left(
	        -\left(\frac{x}{\ell}\right)^2
	    \right),
	\]
	and this can be regarded as the remainder term of the Maclaurin expansion, so
	there is a $\theta\in(0, 1)$ such that
	\[
	    k_n(x, x) = \frac1{(n-1)!}\exp\left(-\theta\left(\frac{x}\ell\right)^2\right)\left(\frac{x}\ell\right)^{2(n-1)}
	    \le \frac{(x/\ell)^{2(n-1)}}{(n-1)!}.
	\]
	In particular, if we have $\lvert x\rvert \le R$ for $\mu$-almost all $x$,
	we have a factorial decay
	$\sigma_n \le \frac{(R/\ell)^{2(n-1)}}{(n-1)!}$.
	
	
	\subsection{Uniform Robustness}\label{app:robustness}
	In applications, the RKHS $\mathcal{H}_k$ may be mis-specified and the quadrature rule $\mu^Q$ when computed for the mis-specified function class $\mathcal{H}_k$ but applied to a function $f \notin \mathcal{H}_k$ leads only to the attainable bound 
	\begin{align*}
		&\left|\int_\X f(x) \dd\mu^Q(x) - \int_\X f(x) \dd\mu(x)
		\right|\\
		&\leq \sup_{x \in \X} \bigl\lvert f(x)
		- \tilde{f}(x)\bigr\rvert
		(|\mu^Q|_{\text{TV}}+ |\mu|_{\text{TV}})
		+ \bigl\lVert \tilde{f}\bigr\rVert_{\Hil_k}\wce(\mu^Q; \Hil_k, \mu) 
	\end{align*}
	via triangle equality and standard integral estimates.
	Note that $|\cdot|_{\text{TV}}$ denotes the total variation norm and the above applies to any $\tilde{f}\in \Hil_k$; in particular, to the best approximation in uniform norm to $f$ in $\Hil_k$. 
	Since $\mu$ is a probability measure, $|\mu|_{\text{TV}}=1$ but if $\mu^Q$ is a signed measure with non-convex weights,
	its total variation $|\mu^Q|_{\text{TV}}$ can be large,
	resulting in arbitrary large integration errors.
	
	\subsection{Equivalence between the projection/matrix Nystr{\"o}m approximations}\label{sec:equiv-nys}
	Let $k$ be a positive semi-definite kernel on $\X$,
	$Z = (z_i)_{i=1}^\ell\subset\X$.
	Let $P_Z:\Hil_k\to\Hil_k$ be the projection operator
	onto $\mathop\mathrm{span}\{k(\cdot, z_i)\mid i=1,\ldots,\ell\}$.
	For arbitrary $x, y\in\X$,
	we can write
	\[
	P_Zk(\cdot, x)=\sum_{i=1}^\ell a_ik(\cdot, z_i),
	\qquad
	P_Zk(\cdot, y)=\sum_{i=1}^\ell b_ik(\cdot, z_i),
	\]
	where $a=(a_i)_{i=1}^\ell, b=(b_i)_{i=1}^\ell\in\R^\ell$.
	From the properties of projection,
	we have
	\[
	k(z_j, x) = \ip{k(\cdot, z_j), k(\cdot, x)}_{\Hil_k}
	=\ip{k(\cdot, z_j), P_Zk(\cdot, x)}_{\Hil_k}
	= \sum_{i=1}^\ell a_ik(z_j, z_i).
	\]
	In matrix notation, we have
	$k(Z, x) = k(Z, Z)a$,
	and $k(Z, y) = k(Z, Z)b$ from the same argument.
	Thus, by conbining it with the property of Moore--Penrose inverse,
	we have
	\begin{align*}
		\ip{P_Zk(\cdot, x), P_Zk(\cdot, y)}_{\Hil_k}
		&= a^\top k(Z, Z)b\\
		&= a^\top k(Z, Z)k(Z, Z)^+k(Z, Z)b \tag{Moore--Penrose}\\
		&= k(x, Z)k(Z, Z)^+ k(Z, y).
	\end{align*}
	This shows the desired equivalence.
	
	\section{Proofs}\label{sec:proof-app}
	
	\subsection{Proof of Theorem \ref{thm:unify}}
	Before proceeding to the proof
	of the theorem,
	we prepare a couple of assertions.
	The following is a well-known
	estimate proven by using the Cauchy--Schwarz inequality
	\citep[see e.g.,][Lemma 3.1 and its proof]{mua17}.
	\begin{prop}\label{prop:c-s}
		Let $k$ be a positive semi-definite kernel on $\X$,
		and $\nu$ be a Borel probability measure
		with $\int_\X \sqrt{k(x,x)}\dd\nu(x) < \infty$.
		Then, for each $f\in\Hil_k$, we have
		\[
		\left\lvert
		\int_\X f(x)\dd\nu(x)
		\right\rvert
		\le \lVert f\rVert_{\Hil_k}
		\int_\X\sqrt{k(x, x)}\dd\nu(x).
		\]
	\end{prop}
	
	By using the proposition,
	we obtain the following technical lemma.
	\begin{lem}\label{lem:k_1}
		Let $k$ and $k_1$ be
		a positive semi-definite kernels
		on $\X$
		such that $k - k_1$
		is also positive semi-definite.
		Let $\nu$ be a Borel probability measure
		on $\X$.
		Then, for any $n\ge1$,
		$a_1,\ldots, a_n\in\R$, $x_1,\ldots, x_n\in\X$,
		if we let $f = \sum_{i=1}^na_ik(\cdot, x_i)$
		and $f_1 = \sum_{i=1}^n a_i k_1(\cdot, x_i)$,
		then we have
		\[
		\left\lvert\int_\X f_1(x)\dd\nu(x)\right\rvert
		\le \lVert f\rVert_{\Hil_k}
		\int_\X \sqrt{k_1(x, x)}\dd\nu(x).
		\]
	\end{lem}
	\begin{proof}
		From the positive semi-definiteness of
		$k_0 \coloneqq k-k_1$, we have
		\begin{align*}
			\lVert f_1\rVert_{\Hil_{k_1}}^2
			= \sum_{i,j=1}^n a_ia_jk_1(x_i, x_j)
			&\le \sum_{i,j=1}^n a_ia_jk_1(x_i, x_j)
			+ \sum_{i,j=1}^n a_ia_jk_0(x_i, x_j)\\
			&= \sum_{i,j=1}^n a_ia_jk(x_i, x_j)
			= \lVert f \rVert_{\Hil_k}^2.
		\end{align*}
		Hence it suffices to prove
		$\lvert\nu(f_1)\rvert \le \lVert f_1\rVert_{\Hil_{k_1}}\nu(g)$ for
		$g(x) \coloneqq \sqrt{k_1(x, x)}$,
		but it directly follows from Proposition
		\ref{prop:c-s}.
	\end{proof}

	\begin{proof}[Proof of Theorem \ref{thm:unify}]
		Note first that, for each $f\in\Hil_{k_0}$,
		$f$ is integrable with respect to $\mu$.
		Indeed, we have
		\[
		\lvert f(x)\rvert
		=\lvert\langle f, k_0(\cdot, x)\rangle_{\Hil_{k_0}}
		\rvert
		\le \lVert f\rVert_{\Hil_{k_0}}
		\lVert k_0(\cdot, x)\rVert_{\Hil_{k_0}}
		= \lVert f\rVert_{\Hil_{k_0}}\sqrt{k_0(x, x)}
		\le \lVert f\rVert_{\Hil_{k_0}}\sqrt{k(x, x)},
		\]
		and it is integrable from assumption,
		so the equality $Q_n(f) = \mu(f)$
		with $f = k_0(\cdot, x)$ is attained at a finite value.
		
		Once we establish \eqref{eq:unify},
		the item (b) is clear, and (a) follows from the fact that
		$Q_n(g)$ and $\mu(g)$ are both integrals
		of the function $g$ with respect to a probability measure. Also, (c) is justified as follows:
		\[
		\wce(Q_n; \Hil_k, \mu)^2
		\le (Q_n(g)+ \mu(g))^2
		\le 2Q_n(g)^2 + 2\mu(g)^2
		\le 2Q_n(g^2) + 2\mu(g^2) \le 4 \mu(g^2),
		\]
		where $\mu(Q_n)^2\le \mu(Q_n^2)$ and
		$\mu(g)^2\le \mu(g^2)$ follows from
		the Cauchy--Schwarz.
		
		To prove \eqref{eq:unify},
		we first prove
		\begin{equation}
			\lvert Q_n(f) - \mu(f) \rvert
			\le \lVert f\rVert_{\Hil_k}(Q_n(g) + \mu(g))
			\label{eq:fin-comb}
		\end{equation}
		for any $f$ of the form
		$f = \sum_{i=1}^na_ik(\cdot, x_i)$
		with $n\ge0$ and $a_1,\ldots,a_n\in\R$.
		Given such an $f$,
		we have $Q_n(f_0) = \mu(f_0)$
		for $f_0  \coloneqq  \sum_{i=1}^na_ik_0(\cdot, x_i)$
		from the assumption.
		Thus, by letting
		$f_1  \coloneqq  f - f_0 = \sum_{i=1}^na_ik_1(\cdot, x_i)$,
		we have
		\[
		Q_n(f) - \mu(f) = (Q_n(f) - \mu(f))
		- (Q_n(f_0) - \mu(f_0)) = Q_n(f_1) - \mu(f_1).
		\]
		As we have
		$\lvert\nu(f_1)\rvert \le \lVert f\rVert_{\Hil_k}\nu(g)$
		for $\nu = Q_n, \mu$
		from Lemma \ref{lem:k_1},
		we obtain
		$\lvert Q_n(f_1) - \mu(f_1)\rvert
		\le \lVert f \rVert_{\Hil_k}(Q_n(g) + \mu(g))$,
		and so \eqref{eq:fin-comb}
		is shown for $f$ of the form
		$f = \sum_{i=1}^na_ik(\cdot, x_i)$.
		
		Finally, we generalize \eqref{eq:fin-comb}
		to any $f\in\Hil_k$.
		Let $\tilde{f}\in\Hil_k$ can be written
		in the form $\sum_{i=1}^na_ik(\cdot, x_i)$.
		If we let $h(x) = \sqrt{k(x, x)}$,
		from Proposition \ref{prop:c-s},
		we have
		\[
		\lvert Q_n(f - \tilde{f})\rvert
		\le \lVert f-\tilde{f}\rVert_{\Hil_k}Q_n(h),
		\qquad
		\lvert \mu(f - \tilde{f})\rvert
		\le \lVert f-\tilde{f}\rVert_{\Hil_k}\mu(h).
		\]
		Note that $\mu(h)<\infty$
		follows from the integrability of $k(x, x)$
		in Assumption \ref{asp:ker}.
		Therefore, we have
		\begin{align*}
			\lvert Q_n(f) - \mu(f) \rvert
			&\le \lvert Q_n(\tilde{f}) - \mu(\tilde{f})\rvert + 
			\lvert Q_n(f-\tilde{f}) - \mu(f-\tilde{f})\rvert
			\\
			&\le
			\lVert \tilde{f}\rVert_{\Hil_k}(Q_n(g) + \mu(g)) +
			\lVert f-\tilde{f}\rVert_{\Hil_k}(Q_n(h) + \mu(h))
			\\
			&\le
			\lVert f \rVert_{\Hil_k}(Q_n(g) + \mu(g)) + 
			\lVert f-\tilde{f}\rVert_{\Hil_k}
			(Q_n(g) + \mu(g) + Q_n(h) + \mu(h)).
		\end{align*}
		As we can make $\lVert f-\tilde{f}\rVert_{\Hil_k}$
		arbitrarily small from the definition of $\Hil_k$,
		the proof of \eqref{eq:unify} is completed
		by taking the limit.
	\end{proof}
	
	\subsection{Proof of Theorem \ref{thm:emp-unify}}
	\begin{proof}
		Denote $D_N=\{y_1,\ldots,y_N\}$ and note that the result follows
		from \eqref{eq:emp-unify}
		and
		\begin{equation}
			\E{\tilde{\mu}_N(g)^2}
			=\E{\tilde{\mu}_N(g^2)}
			\le
			\E{
				\frac1N\sum_{i=1}^Ng(y_i)^2
			}
			= \int_\X k_1(x, x)\dd\mu(x),
			\label{eq:emp-surro}
		\end{equation}
		where the first inequality is given by
		the Cauchy--Schwarz.
		
		Indeed, (a) is an immediate consequence
		of \eqref{eq:emp-unify}
		and $Q_n$ and $\tilde{\mu}_N$
		making a probability measure,
		and (b) is obtained as
		$   2\E{(Q_n(g) + \tilde{\mu}_N(g))^2}
		\le 8\E{\tilde{\mu}_N(g)^2}
		\le 8\int_\X k_1(x, x)\dd\mu(x)$
		by using \eqref{eq:emp-surro} and
		the requirement $Q_n(g)\le \tilde{\mu}_N(g)$.
		
		When the requirement is
		$Q_n(g^2)\le \tilde{\mu}_N(g^2)$,
		as we have $Q_n(g)^2\le Q_n(g^2)$
		and $\tilde{\mu}_N(g)^2\le \tilde{\mu}_N(g^2)$
		by the Cauchy--Schwarz, we also have by the AM--GM,
		\begin{align*}
			2\E{(Q_n(g) + \tilde{\mu}_N(g))^2}
			&\le 4\E{Q_n(g)^2} + 4\E{\tilde{\mu}_N(g)^2}\\
			&\le 4\E{Q_n(g^2)} + 4\E{\tilde{\mu}_N(g^2)}\\
			&\le 8\E{\tilde{\mu}_N(g^2)}
			\le 8\int_\X k_1(x, x)\dd\mu(x)
		\end{align*}
		
		For showing \eqref{eq:emp-unify},
		we remark that
		we always have
		\begin{equation}
			\wce(Q_n; \Hil_k, \tilde{\mu}_N)
			\le Q_n(g) + \tilde{\mu}_N(g)
			\label{eq:q_n+mu_n}
		\end{equation}
		by applying Theorem \ref{thm:unify}
		with $\tilde{\mu}_N$ instead of $\mu$.
		
		Let
		$h(\mu), h(\tilde{\mu}_N), h(Q_n) \in \Hil_k$ be
		the kernel mean embeddings of
		$\mu$, $\tilde{\mu}_N$ and $\mu^{Q_n}$,
		i.e.,
		\[
		h(\mu)  \coloneqq  \int_\X k(\cdot, x)\dd\mu(x),
		\quad
		h(\tilde{\mu}_N)
		\coloneqq  \frac1N\sum_{i = 1}^N k(\cdot, y_i),
		\quad
		h(Q_n)
		\coloneqq  \sum_{i=1}^n w_ik(\cdot, x_i),
		\]
		where $(w_i)_{i=1}^n$ and $(x_i)_{i=1}^n$
		are weights and points defining
		the quadrature $Q_n$.
		Remark that $h(\mu)$ is well-defined
		as $\int_\X k(x, x)\dd\mu(x)<\infty$
		\citep[][Lemma 3.1]{mua17}.
		As we can rewrite the worst-case error as
		\[
		\wce(Q_n;\Hil_k,\tilde{\mu}_N)
		= \lVert
		h(Q_n) - h(\tilde{\mu}_N)
		\rVert_{\Hil_k},
		\qquad
		\wce(Q_n;\Hil_k,\mu)
		= \lVert
		h(Q_n) - h(\mu)
		\rVert_{\Hil_k},
		\]
		by triangle inequality and the AM--GM,
		we obtain
		\begin{align*}
			\E{\wce(Q_n;\Hil_k,\mu)^2}
			&\le \E{(\wce(Q_n;\Hil_k,\tilde{\mu}_N)
				+ \lVert h(\mu) - h(\tilde{\mu}_N) \rVert_{\Hil_k})^2}\\
			&\le 2\E{\wce(Q_n;\Hil_k,\tilde{\mu}_N)^2}
			+ 2\E{\lVert h(\mu) - h(\tilde{\mu}_N) \rVert_{\Hil_k}^2}\\
			&\le 2\E{(Q_n(g) + \tilde{\mu}_N(g))^2}
			+ 2\E{\lVert h(\mu) - h(\tilde{\mu}_N) \rVert_{\Hil_k}^2},
		\end{align*}
		where we have used \eqref{eq:q_n+mu_n}
		in the last inequality.
		It now suffices to prove
		$\E{\lVert h(\mu) - h(\tilde{\mu}_N) \rVert_{\Hil_k}^2} = c_{k, \mu}/N$
		for showing \eqref{eq:emp-unify}.
		
		Indeed, we have
		\begin{align*}
			&\E{\lVert h(\mu) - h(\tilde{\mu}_N) \rVert_{\Hil_k}^2}
			= \E{\lVert
				h(\mu)
				\rVert_{\Hil_k}^2}
			-2
			\E{\ip{h(\mu), h(\tilde{\mu}_N)}_{\Hil_k}}
			+
			\E{\lVert
				h(\mu)
				\rVert_{\Hil_k}^2}
			\\
			&=
			\iint_{\X\times\X}k(x, y)\dd\mu(x)\dd\mu(y)
			- \frac2N \sum_{i=1}^N \int_\X \E{k(x, y_i)}\dd\mu(x)
			+
			\frac1{N^2}\sum_{i,j=1}^N \E{k(y_i, y_j)}
			\\
			&=\frac1{N^2}\sum_{i=1}^N\E{k(y_i, y_i)}
			+\left(1 - 2 + \frac{N(N-1)}{N^2} \right)\iint_{\X\times\X}k(x, y)\dd\mu(x)\dd\mu(y)
			=\frac{c_k}N,
		\end{align*}
		since $\iint_{\X\times\X}k(x, y)\dd\mu(x)\dd\mu(y)
		= \int_\X\E{k(x, y_i)}\dd\mu(x)
		= \E{k(y_i, y_j)}$ holds for $i\ne j$.
		Thus, the proof is completed.
	\end{proof}
	
	\subsection{Proof of Proposition \ref{prop4}}
	\begin{proof}
		As $Q_n$ exactly integrates the functions in $\Hil_{k_0}$,
		we have $\wce(Q_n; \Hil_{k_0}, \mu) = 0$.
		So, if we set $Q_n(f) = \sum_{i=1}^n w_if(x_i)$,
		then we have, from \eqref{eq:wce-general}
		with kernel $k_0$,
		\begin{equation}
			0 = \sum_{i,j=1}^n w_iw_jk_0(x_i, x_j)
			-2 \sum_{i=1}^n w_i\int_\X k_0(x_i, y)\dd\mu(y)
			+ \iint_{\X\times\X} k_0(x, y)\dd\mu(x)\dd\mu(y).
			\label{eq:k0-wce}
		\end{equation}
		If we extract this from the formula \eqref{eq:wce-general}
		for the kernel $k$,
		we have, by letting $k_1  \coloneqq  k - k_0$,
		\begin{align*}
			\wce(Q_n; \Hil_k, \mu)^2
			&= \wce(Q_n; \Hil_k, \mu)^2 - \wce(Q_n; \Hil_{k_0}, \mu)^2\\
			&= \sum_{i,j=1}^n w_iw_jk_1(x_i, x_j)
			-2 \sum_{i=1}^n w_i\int_\X k_1(x_i, y)\dd\mu(y)\\
			&\quad + \iint_{\X\times\X} k_1(x, y)\dd\mu(x)\dd\mu(y).
		\end{align*}
		So, if we define
		$M \coloneqq \sup_{x\in\X}\lvert k_1(x, y)\rvert
		= \sup_{x\in\X}\lvert k(x, y) - k_0(x, y)\rvert$,
		we have
		\[
		\wce(Q_n; \Hil_k, \mu)^2
		\le \left(\sum_{i,j=1}^nw_iw_jM
		+ 2\sum_{i=1}^nw_i\int_\X M\dd\mu(y)
		+\iint_{\X\times\X}M\dd\mu(x)\dd\mu(y)
		\right)=4M,
		\]
		as $Q_n$ is a convex quadrature.
		The existence follows from almost the same proof as
		in the proof of Theorem \ref{thm:unify},
		but in this case it directly follows from
		Tchakaloff's thorem \citep{tch57,bay06}.
	\end{proof}
	
	\subsection{Proof of Theorem \ref{thm:uni-exist}}\label{sec:proof:uni-exist}
	\begin{proof}
		We prove the existence of the
		version $Q_n(g)\le\mu(g)$.
		The other follows just by
		replacing every $g$ in the proof below by $g^2$.
		
		Let $\phi_1, \ldots, \phi_{n-1}\in \Hil_{k_0}$
		satisfy $\Hil_{k_0} = \mathop\mathrm{span}
		\{\phi_1, \ldots, \phi_{n-1}\}$.
		Also, let $y, y_1, y_2, \ldots$ be independent samples from $\mu$.
		Now, consider the vector-valued function
		$\bm\psi = (\phi_1, \ldots, \phi_{n-1}, g)^\top\in\R^n$.
		Note that $\E{\lVert\bm\psi(y)\rVert}<\infty$
		follows from the integrability of elements in $\Hil_{k_0}$
		and $g$ with respect to $\mu$.
		Therefore, by
		\cite[][Theorem 11]{hayakawa-MCCC},
		with probability $1$,
		there exists an $N$ such that
		$\E{\bm\psi(y)} \in
		\cv\{\bm\psi(y_1), \ldots, \bm\psi(y_N)\}$.
		So, in particular, there exist deterministic points
		$x_1, \ldots, x_N\in \X$ satisfying
		$\E{\bm\psi(y)} \in
		\cv\{\bm\psi(x_1), \ldots, \bm\psi(x_N)\}$.
		For such $(x_i)_{i=1}^N$,
		consider an optimal solution that is also
		a {\it basic} feasible solution of
		the following linear programming problem:
		\begin{equation}
			\begin{array}{rl}
				\text{minimize} & \displaystyle\sum_{i=1}^N w_i
				g(x_i) \\
				\text{subject to} & 
				\left[
				\bm\phi(x_1) \cdots \bm\phi(x_N)
				\right]
				\displaystyle\bm{w} = \int_\X\bm\phi(x)\dd\mu(x),\ \bm{w}\ge\bm0,
			\end{array}
			\label{eq:lp--}
		\end{equation}
		where $\bm\phi =
		(1, \phi_1, \ldots, \phi_{n-1})^\top\in\R^n$
		is another vector-valued function
		(note that its first coordinate is constant so that
		any feasible solution of \eqref{eq:lp--} sums up to one).
		Such a basic solution $\bm{w}$ has at most
		$n$ nonzero entries, say
		$(w_{i_1}, \ldots, w_{i_n})\in \Delta^n$ with
		$1\le i_1 < \cdots < i_n \le N$.
		Then,
		the quadrature $Q_n$ given by weights $(w_{i_j})_{j=1}^n$
		and points $(x_{i_j})_{j=1}^n$
		satisfies $Q_n(\bm\phi) = \mu(\bm\phi)$
		and $Q_n(g)\le \mu(g)$.
		The latter follows from the optimality of
		$\bm{w}$ and the fact that $\E{\psi(y)}\in\cv\{
		\psi(y_1), \ldots, \psi(y_N)\}$
		leads to a feasible solution with the objective
		$\E{g(y)} = \mu(g)$).
	\end{proof}
	
	\subsection{Proof of Theorem \ref{thm:nys-op-norm}}
	\label{sec:proof-nys}
	We prove the theorem
	by using an existing bound
	regarding the Nyst{\"o}m approximation for matrices,
	which is more common in the machine learning literature.
	
	Let $A = (A_{ij})_{i,j=1}^N
	\in\R^{N\times N}$ be a symmetric
	positive semi-definite matrix.
	Let us denote it as
	$A = [\bm{a}_1, \ldots, \bm{a}_N]$ by using
	$\bm{a}_1,\ldots,\bm{a}_N\in\R^N$.
	Then, we independently sample $i_1, \ldots, i_\ell$
	from $\{1, \ldots, N\}$ uniformly,
	and construct a submatrix
	$B = (a_{i_ji_k})_{j,k=1}^\ell$.
	If we let $B_s$ be the best rank-$s$ approximation
	of $B$ and $B_s^+$ its pseudoinverse,
	the matrix
	\begin{equation}
		\tilde{A}
		= [\bm{a}_{i_1}, \ldots, \bm{a}_{i_\ell}]
		B_s^+
		\left[
		\begin{array}{c}
			\bm{a}_{i_1}\\
			\vdots \\
			\bm{a}_{i_\ell}
		\end{array}
		\right]
		\label{eq:mat-nys}
	\end{equation}
	works as a rank-$s$ approximation of $A$.
	
	We use the following result
	on this matrix version:
	\begin{prop}[{\citep[][Theorem 2]{kum12}}]
		\label{prop-kum-mat}
		For a positive semi-definite matrix $A$,
		the rank-$s$ approximation $\tilde{A}$
		given above satisfies,
		with probability at least $1-\delta$,
		the following:
		\[
		\lVert A - \tilde{A} \rVert_2
		\le
		\lVert A - A_s \rVert_2
		+\frac{2N}{\sqrt{\ell}}
		A_{\max} \left(
		1 + \sqrt{
			\frac{
				D^A_{\max}
			}{A_{\max}}
			\frac{N-\ell}{N-1/2}
			\frac1{\beta(\ell, N)}
			\log\frac1\delta
		}
		\right),
		\]
		where $\beta(\ell, N) =
		1 - \frac1{2\max\{\ell, N-\ell\}}$,
		$A_{\max} = \max_iA_{ii}$,
		$D^A_{\max} = \max_{i,j}
		(A_{ii} + A_{jj} - 2A_{ij})$
		and $A_s$ is the best rank-$s$ approximation
		of $A$.
	\end{prop}
	
	As $D^A_{\max} \le 2 A_{\max}$,
	if we have $N \ge 2\ell$, it holds that
	\[
	\frac{D^A_{\max}
	}{A_{\max}}
	\frac{N-\ell}{N-1/2}
	\frac1{\beta(\ell, N)}
	\le
	2\frac{N-\ell}{N-1/2}\frac{N-\ell-1/2}{N-\ell}\le 2,
	\]
	and we can just state
	\begin{equation}
		\lVert A - \tilde{A} \rVert_2
		\le
		\lVert A - A_s \rVert_2
		+\frac{2N}{\sqrt{\ell}}
		A_{\max} \left(
		1 + \sqrt{
			2
			\log\frac1\delta
		}
		\right).
		\label{eq:nys-mat}
	\end{equation}
	
	We show the following lemma
	as a consequence of this proposition.
	\begin{lem}
		Let $s \le \ell$ be positive integers and
		$\delta > 0$.
		Let $k:\X\times\X$
		be a symmetric and positive semi-definite kernel
		and $y_1, y_2, \ldots$ be i.i.d.~random variables
		taking values in $\X$.
		For each $N$, define
		the $N\times N$ matrices
		$K(N), K_s(N), K_s^Z(N)$
		by
		\[
		K(N)_{ij} = \frac{k(y_i, y_j)}N,
		\quad
		K_s(N)_{ij} = \frac1N \sum_{m = 1}^s \sigma_m
		e_m(y_i)e_m(y_j),
		\quad
		K^Z_s(N)_{ij} = \frac{k^Z_s(y_i, y_j)}N,
		\]
		where $Z = (y_1, \ldots, y_\ell)$.
		
		Then, there exists a sequence $\ve_N \to 0$
		such that
		\begin{equation}
			\lVert K(N) - K^Z_s(N) \rVert_2
			\le
			\lVert K(N) - K_s(N) \lVert_2
			+ \frac{2\sup_{x}k(x, x)}{\sqrt{\ell}}
			\left( 1 + \sqrt{2\log \frac1\delta} \right)
			\label{kum-ineq}
		\end{equation}
		is met with probability at least
		$1 - \delta - \ve_N$.
	\end{lem}
	\begin{proof}
		We assume $N\ge2\ell$.
		Let $i_1, \ldots, i_\ell$ be
		independent uniform samples from
		$\{1, \ldots, N\}$.
		Consider the event $E_N$ that $i_1,\ldots,i_\ell$
		are all different.
		Then, $\P{E_N}=\prod_{i=1}^\ell\frac{N+1-i}N$
		converges to $1$ as $N\to\infty$,
		and let $\ve_N = 1 - \P{E_N}$.
		By using Proposition
		\ref{prop-kum-mat}, \eqref{kum-ineq}
		and $\max_i K(N)_{ii} \le N^{-1}\sup_{x}k(x,x)$,
		we have that the probability
		\[
		\P{
			\lVert K(N) - \tilde{K}_s(N) \rVert_2
			\le
			\lVert K(N) - K_s(N) \lVert_2
			+ \frac{2\sup_{x}k(x, x)}{\sqrt{\ell}}
			\left( 1 + \sqrt{2\log \frac1\delta} \right)
			\lmid E_N
		}
		\]
		is at least $(1-\delta-\ve_N) / \P{E_N}
		\ge 1 - \delta - \ve_N$,
		where $\tilde{K}_s(N)$
		is the rank-$s$ Nystr{\"o}m approximation
		of the matrix $K(N)$ by using indices
		$i_1, \ldots,i_\ell$.
		From \eqref{eq:mat-nys},
		if we take
		$\tilde{W} = k(y_{i_j}, y_{i_k})_{j,k=1}^\ell$
		and $\tilde{W}_s$ its best rank-$s$ approximation,
		it actually satisfies
		\[
		\tilde{K}_s(N)_{ij}
		= \frac1N k(y_i, D)\tilde{W}^+_sk(D, y_j)
		= \frac1N k^D_s(y_i, y_j),
		\]
		where $D = (y_{i_1}, \ldots, y_{i_\ell})$
		and $k_s^D$ is the Nystr{\"o}m
		approximation given in the main body.
		
		As $y_1, \ldots, y_N$ are i.i.d.~samples,
		we can see that
		$(Z, (y_i)_{i=1}^N)$ (without any conditioning)
		and
		$(D, (y_i)_{i=1}^N)$
		conditioned on $E_N$ actually
		have the same distribution,
		so we are done.
	\end{proof}
	
	We finally prove the result for
	the Nystr{\"o}m approximation
	of integral operators.
	
	\begin{proof}[Proof of Theorem \ref{thm:nys-op-norm}]
		Take a sufficiently large $N$
		and let us use $K(N), K_s(N), K_s^Z(N)$
		defined in the previous lemma
		with
		$y_1, y_2, \ldots$ independently sampled from $\mu$.
		
		It suffices to
		consider the case
		$C_k  \coloneqq  \sup_{x\in \X} k(x, x) < \infty$.
		It is clear that
		$K_s(N)_{ii} \le K(N)_{ii} \le C_k / N$,
		and from \eqref{nys-s}, we also have
		\[
		k_s^Z(x, x) = k(x, Z)W^+_s k(Z, x)
		\le k(x, Z) W^+ k(Z, x)
		= \lVert P_Zk(\cdot, x)\rVert_{\Hil_k}^2
		\le \lVert k(\cdot x)\rVert_{\Hil_k}^2
		= k(x, x),
		\]
		and so $K^Z_s(N)_{ii} \le C_k / N$.
		
		For a matrix $A(N)\in\R^{N\times N}$
		defined by $A(N)_{ij} = (1-\delta_{ij})(K(N) - K_s^Z(N))$,
		i.e., the matrix given by deleting the diagonal,
		we have
		$\lVert A(N) \rVert_2 \to
		\lVert \mathcal{K}_s^Z - \mathcal{K} \rVert$
		as $N \to\infty$ almost surely \citep[Theorem 3.1]{kol00}.
		Since we have observed that
		$\lVert K(N) - K_s^Z(N) - A(N) \rVert_2 \le C_k / N$,
		we have
		\[
		\lVert K(N) - K_s^Z(N) \rVert_2 \to
		\lVert \mathcal{K}_s^Z - \mathcal{K} \rVert,
		\qquad
		N\to\infty
		\]
		almost surely.
		The same argument yields
		$\lVert K(N) - K_s(N) \rVert_2 \to \sigma_{s+1}$,
		as it converges to the norm of the integral operator
		given by the kernel
		$\sum_{m\ge s+1}\sigma_me_m(x)e_m(y)$.
		
		Now, by letting $A_N$ be the event
		that \eqref{kum-ineq} holds (so $\P{A_N} \ge 1 - \delta - \ve_N$),
		the desired inequality \eqref{wtp1} almost surely holds under
		the event $\limsup A_N
		= \bigcap_{N > \ell}\bigcup_{M \ge N} A_M$.
		Indeed, under this event we can just take
		the limit of both sides of \eqref{kum-ineq}
		for an appropriate subsequence of
		$(2\ell, 2\ell+1, \ldots)$.
		As we have
		\[
		\P{\limsup A_N}
		= \lim_{N\to\infty}
		\P{\bigcup_{M \ge N}A_N} \ge
		\lim_{N\to\infty} (1 - \delta - \ve_N)
		= 1 - \delta,
		\]
		the proof is completed.
	\end{proof}

	\section{Kernel Quadrature when Expectations are Known}\label{app:known expectations}
	When we use an approximate kernel $k_0$
	and know exact expectation of test functions
	$\phi_1, \ldots, \phi_{n-1}$
	with
	$\Hil_{k_0}\subset
	\mathop\mathrm{span}\{\phi_1,\ldots,\phi_{n-1}\}$,
	we can obtain an $n$-point kernel quadrature
	that exactly integrates $\phi_1,\ldots,\phi_{n-1}$
	by Algorithm \ref{algo:rchq}.
	\begin{algorithm}[H]
		\caption{Kernel Quadrature with Random Convex Hulls}\label{algo:rchq}
		\begin{algorithmic}[1]
			\Require{A positive semi-definite kernel $k$ on $\X$, a probability measure $\mu$ on $\X$, integers $N \ge n \ge1$,
				another kernel $k_0$
				and functions $\phi_1,\ldots,\phi_{n-1}$
				on $\X$
				with $\Hil_{k_0}\subset \mathop\mathrm{span}
				\{\phi_1,\ldots,\phi_{n-1}\}$}
			\Ensure{With some probability,
				returns $Q_n  \coloneqq  \{(w_i,x_i)\mid i=1,\ldots,n\}\subset \R \times \X$
				with $(w_i)\in\Delta^n$
			}
			\State{Calculate the expectations $\int_\X\phi_1(x)\dd\mu(x), \ldots, \int_\X\phi_{n-1}(x)\dd\mu(x)$}\label{state:expectations}
			\State{Sample $y_1,\ldots,y_N$
				independently from $\mu$}\label{state:sample}
			\State{For a vector-valued function
				$\bm\phi = (\phi_1, \ldots, \phi_{n-1})^\top$
				and $k_{1,\mathrm{diag}}(x)
				=k(x, x)-k_0(x, x)$,
				solve the linear programming problem
				($\lvert\cdot\rvert_0$ denotes the number of nonzero entries)
				\begin{equation}
					\begin{array}{rl}
						\text{minimize} & \bm{w}^\top
						k_{1, \mathrm{diag}}(\bm{x}) \\
						\text{subject to} & 
						\left[
						\bm\phi(y_1) \cdots \bm\phi(y_N)
						\right]
						\displaystyle\bm{w} = \int_\X\bm\phi(x)\dd\mu(x),\\
						& \bm{w}\ge\bm0,\ \bm1^\top\bm{w} = 1,
						\ \lvert\bm{w}\rvert_0\le n.
					\end{array}
					\label{eq:lp}
				\end{equation}
				to obtain points $\{x_1, \ldots, x_n\}\subset
				\{y_1, \ldots, y_N\}$
				and weights $(w_i)\in\Delta^n$ satisfying
				\[
				\sum_{i=1}^nw_i\bm\phi(x_i) = \int_\X \bm\phi(x) \dd\mu(x)
				\]}\label{state:LP}
			if \eqref{eq:lp} is feasible.
		\end{algorithmic}
	\end{algorithm}
	We make several remarks on this algorithm.
	First, the problem \eqref{eq:lp}
	is, strictly speaking, not a linear programming (LP),
	as it includes the sparsity constraint
	$\lvert\bm{w}\rvert_0\le n$.
	However, as it only contains $n$ equality constraints,
	its {\it basic feasible solution} always satisfies
	$\lvert\bm{w}\rvert_0\le n$ and
	the simplex algorithm automatically gives
	such a sparse (and optimal)
	solution even if we do not explicitly
	impose this constraint,
	so we call it an LP for simplicity.
	Second, {\it this algorithm occasionally fails
		to output $Q_n$}
	as, with some probability,
	the LP has no feasible solution.
	Although we can repeat the algorithm
	until we succeed, the number $N$ should
	be chosen appropriately.
	See Remark \ref{rem:mccc} for this point.
	Finally, our algorithm has possibly related
	approaches such as sparse optimization
	and Sard's method,
	see Remark \ref{rem:subsampling} and \ref{rem:sard}.
	
	\begin{remark}\label{rem:mccc}
		A simple approach for constructing a quadrature formula
		\citep{hayakawa-MCCC} was recently proposed:
		randomly sample candidate points and find a solution by using a linear programming (LP) solver.
		Indeed, for an independent sample $y_1, \ldots, y_N \sim \mu$,
		we can construct a quadrature formula with convex weights
		exactly integrating the functions in
		$\F = \mathop\mathrm{span}\{\phi_1, \ldots, \phi_{n-1}\}$ using a subset of these points
		if and only if we have
		\begin{equation}
			\int_\X\bm\phi(x)\dd\mu(x) \in \cv \{ \bm\phi(y_1), \ldots, \bm\phi(y_N)\},
			\label{eq-cv}
		\end{equation}
		where $\bm\phi=(\phi_1, \ldots, \phi_{n-1})^\top:\X\to\R^{n-1}$
		and $\cv A$ denotes the convex hull of $A$.
		Several sharp estimates for the probability of the event \eqref{eq-cv}
		are available in \citet{hayakawa21a}.
		Under the event \eqref{eq-cv}, we can find a desired rule by using the simplex method 
		for the LP problem \eqref{eq:lp}.
	\end{remark}
	\begin{remark}\label{rem:subsampling}
		From the viewpoint of subsampling, a direct way to obtain quadrature formulas with convex weights
		supported on a small number of points, is to first sample $N$ candidate points $D_N = (x_1, \ldots, x_N)$ and then solve the following sparse optimization problem:
		\begin{equation}
			\begin{array}{rl}
				\text{minimize} & \bm{w}^\top k(D_N, D_N) \bm{w}
				- 2\bm{w}^\top \int_\X k(D_N, y)\dd\mu(y) \\
				\text{subject to} & 
				\bm{w}\ge\bm0,
				\ \bm{1}^\top\bm{w} = 1,\ \lvert\bm{w}\rvert_0 \le n,
			\end{array}
			\label{sparse-qp}
		\end{equation}
		where $k(D_N, D_N)$ is the corresponding $N \times N$ Gram matrix.
		Unfortunately, exactly solving this problem is computationally challenging, in particular in contrast to our approach that exploits the spectral properties of $k$ and $\mu$.
		Nevertheless, one could use sparse optimization to obtain an approximate solution of \eqref{sparse-qp}: although the simplex constraint ($\bm{w}\ge\bm0$, $\bm{1}^\top\bm{w} = 1$) makes it impossible to exploit the classical $\ell_1$ regulatization, there are possible alternatives under this constraint \citep{pil12,kyr13,li20} or use the DC (difference of convex functions) algorithm to incorporate the sparsity constraint to find a local minima \citep{got18}.
		This is a promising research direction, and our general sample estimates might provide a first step towards this direction.
	\end{remark}
	\begin{remark}\label{rem:sard}
		Sard's method \citep{sar49,lar70} for constructing numerical integration rules
		uses the $n$ degree of freedom (of choosing weights in our setting)
		separately;
		$m$ ($\le n$) for exactness over a certain $m$-dimensional space
		of test functions,
		and the remaining $n - m$ for minimizing an error criterion
		such as the worst-case error.
		In the context of kernel quadrature,
		one way to use Sard's method with exactness over $\F$
		(an $m$-dimensional space of test functions) is as follows
		\citep{kar18,sou20}:
		\begin{equation}
			\begin{array}{rl}
				\text{minimize} & \bm{w}^\top k(D_n, \bm{x}_n) \bm{w}
				- 2\bm{w}^\top \int_\X k(\bm{x}_n, y)\dd\mu(y) \\
				\text{subject to} & 
				\bm{w}^\top f(D_n) = \int_\X f(y)\dd\mu(y),\ \forall f\in\F,
			\end{array}
			\label{sard}
		\end{equation}
		where $D_n = (x_1, \ldots, x_n)^\top$,
		$f(D_n) = (f(x_1), \ldots, f(x_n))^\top$.
		This amounts to solving a convex quadratic programming for $\bm{w}$ in an $(n - m)$-dimensional subspace of $\R^n$ (without constraint).
		This is similar to our approach in that it enforces exactness in a certain finite-dimensional space of test functions.
		One key difference is that Sard's approach aims for a quadrature formula on a given set of points, whereas our method determines also the points themelves.
		Hence, the combination of these two approaches seems to be an interesting future research topic.\footnote{For example, we can pick the first $m$ eigenfunctions of the integral operator as test functions, and find $n$ points and weights that minimizes the worst-case error while exactly integrating the test functions from a larger set of candidate points.
			An obvious challenge is that a quadratic programming does not supply sparsity, whereas the approach of this paper has been fully based on the sparsity of a basic feasible solution of an LP problem.}
	\end{remark}
	
	\paragraph{Computational complexity.}
	A tricky part of this approach,
	essentially based on random convex hulls,
	is that the algorithm possibly
	does not output a quadrature formula.
	Hence, the following quantity
	plays an important role to estimate
	the essential complexity of the algorithm:
	\[
	N_\phi = \inf\left\{N\ge1\lmid
	\P{
		\E{\bm\phi(y)}
		\in\cv\{\bm\phi(y_1),
		\ldots,\bm\phi(y_N)\}}\ge\frac12\right\},
	\]
	where $y, y_1, y_2,\ldots$ are independent samples
	from $\mu$.
	This value is known to be finite
	and estimated under a variety of conditions
	on $\bm{\phi}(y)$ \citep{wag01,hayakawa21a}.
	If we have some knowledge of $\mu$,
	we can just keep trying the algorithm with $N=N_\phi$
	until it succeeds,
	and its expected computational time is
	$\ord{nN_\phi + C(n, N_\phi)}$,
	where $C(a, b)$ is the (expected)
	cost of solving an $a\times b$
	LP with a simplex method.
	Note that, though the worst-case computational time
	of the simplex method is exponential,
	it is empirically $\ord{ab\min\{a, b\}}$
	in practice \citep{pan85,sha87}.
	In addition, $N_\phi = \ord{n}$ holds
	in examples with some symmetry
	\citep{wen62,hayakawa-MCCC},
	so in that case
	we have a heuristic complexity estimate
	of $\ord{n^3}$.
	
	\paragraph{Choice of approximate kernels.}
	Similarly to the empirical version
	discussed in the main text,
	we prove quantitative estimates when $k_0$ is given
	by the Mercer approximation
	or Nystr{\"o}m approximation.
	Remark that the necessary information
	for using these methods is different.
	Whereas using the Mercer approximation
	requires the knowledge
	of Mercer decomposition $k(x, y)=\sum_{m=1}^\infty
	\sigma_me_m(x)e_m(y)$
	and their exact integration
	$\int_\X e_m(x)\dd\mu(x)$,
	the Nystr{\"o}m approximation only requires
	the exact integral values of kernel,
	$\int_\X k(x, y)\dd\mu(y)$,
	and so is more generally applicable.
	See the following sections for details.
	
	In the following, we assume that the kernel
	attains the Mercer decomposition
	$k(x, y)=\sum_{m=1}^\infty\sigma_me_m(x)e_m(y)$,
	where $\sigma_1\ge\sigma_2\ge\cdots\ge0$
	and $(e_m)_{m=1}^\infty$ is an orthonormal set
	of $L^2(\mu)$.
	
	\subsection{Algorithm \ref{algo:rchq} with
		Mercer Approximation}\label{sec:uni}
	If we use the truncated Mercer decomposition
	as an approximate kernel,
	we have the following result.
	\begin{thm}\label{thm:app-mer}
		If Algorithm \ref{algo:rchq}
		with $k_0\coloneqq
		\sum_{m=1}^{n-1}\sigma_me_m(x)e_m(y)$
		and $\phi_i= e_i$
		successfully outputs a convex quadrature $Q_n$,
		then it satisfies the following:
		\begin{itemize}
			\item[(a)]
			If $C\coloneqq\sup_{m\ge1}\lVert
			e_m\rVert_\infty < \infty$, we have
			$\wce(Q_n; \Hil_k, \mu)^2 \le 4C^2
			\sum_{m=n}^\infty\sigma_m$.
			\item[(b)]
			As $N$ in Algorithm
			\ref{algo:rchq} tends to infinity,
			we have
			\[
			\P{\wce(Q_n;\Hil_k, \mu)^2
				\le 4\sum_{m=n}^\infty\sigma_m}
			\to1.
			\]
		\end{itemize}
	\end{thm}
	
	\begin{proof}
		As in the proof of Corollary \ref{main-bdd},
		$k-k_0$ is positive semi-definite.
		Thus, when $\sum_{m=n}^\infty\sigma_m<\infty$,
		the kernel
		and the measure $\mu$
		satisfies Assumption \ref{asp:ker}.
		So Thorem \ref{thm:unify}(a) implies (a) of
		this theorem,
		since $k_1(x, x)=
		\sum_{m=n}^\infty\sigma_me_m(x)^2
		\le C^2\sum_{m=n}^\infty\sigma_m$.
		
		For (b), if we have
		$Q_n(k_{1,\mathrm{diag}})
		\le \mu(k_{1,\mathrm{diag}})$,
		then Theorem \ref{thm:unify} implies
		\[
		\wce(Q_n;\Hil_k, \mu)^2
		\le 4\int_\X k_1(x, x)\dd\mu(x)
		=\int_\X\sum_{m=n}^\infty
		\sigma_m e_m(x)^2\dd\mu(x)
		=
		4\sum_{m=n}^\infty\sigma_m.
		\]
		So it suffices to prove
		$\P{Q_n(k_{1,\mathrm{diag}})
			\le \mu(k_{1,\mathrm{diag}})} \to 1$
		as $N\to\infty$,
		and it is shown by considering
		the optimal basic feasible solution of the LP
		\eqref{eq:lp}
		and the following fact
		\citep[][Proposition 4]{hayakawa21a}:
		\[
		\P{\E{\bm\psi(y)}\in\cv\{\bm\psi(y_1),\ldots,
			\psi(y_N)\}}\to 1,
		\qquad N\to\infty,
		\]
		where $\psi = (\phi_1, \ldots, \phi_{n-1},
		k_{1,\mathrm{diag}})^\top$.
		Indeed, under the event $\bm\psi(y)\in\cv\{\bm\psi(y_1),\ldots,
		\psi(y_N)\}$,
		the LP becomes feasible and
		$Q_n(k_{1,\mathrm{diag}})
		\le \mu(k_{1,\mathrm{diag}})$
		follows from the optimality.
		See the proof of Theorem \ref{thm:uni-exist}
		(Section \ref{sec:proof:uni-exist})
		for a more detailed explanation if necessary.
	\end{proof}
	
	Note that the boundedness of $C$ is a typical assumption
	\citep[see][Assumption 3.2 and references therein]{lu16},
	while it does not necessarily hold \citep[Section 3]{min06}.
	Under some assumptions,
	we can quantify the
	probability that the LP \eqref{eq:lp} becomes
	feasible.
	
	\paragraph{Sampling bound.}
	Suppose $1$ is an eigenfunction of $\mathcal{K}$.
	This is satisfied, e.g., in the following
	cases:
	\begin{itemize}
		\item $\mu$ is a Haar measure on a compact group
		and $k$ is shift-invariant.
		\item $k$ is a kernel based on Stein's identity
		\citep{oat17,sou20,ana21} with respect to $\mu$.
	\end{itemize}
	In this case,
	we have a theoretical bound of the required $N$
	in Algorithm \ref{algo:rchq} as follows.
	\begin{thm}
		Suppose $1$ is an eigenfunction of $\mathcal{K}$, i.e.,
		$\int_\X k(\cdot, y)\dd\mu(y)$ is a constant function.
		Then, for each $n\ge2$ and $N \ge 6(n-1)\sup_{x\in\X}\sum_{m=1}^{n-1}e_m(x)^2$,
		Algorithm \ref{algo:rchq} returns
		a feasible quadrature with probability
		at least $1- 2^{1-n}$, i.e.,
		for an independent sample $y_1, \ldots, y_N$ from $\mu$,
		we have
		\[
		\P{\int_\X\bm\phi(x)\dd\mu(x)
			\in \cv\{\bm\phi(y_1), \ldots, \bm\phi(y_N)\}} \ge 1 -\frac1{2^{n-1}},
		\]
		where $\bm\phi = (e_1, \ldots, e_{n-1})^\top$.
		If the value $C = \sup_{m\ge1}\lVert e_m\rVert_\infty$ is finite,
		$N\ge 6C(n-1)^2$ is also sufficient for the above estimate.
	\end{thm}
	\begin{proof}
		This follows from the existing
		results \citep[Theorem 14 and Proposition 17]{hayakawa21a}. 
	\end{proof}
	
	\subsection{Algorithm \ref{algo:rchq}
		with Nystr{\"o}m Approximation}\label{sec:app-nys}
	
	Although the method discussed in the previous section
	requires the knowledge of Mercer decomposition,
	if we make use of the Nystr{\"o}m approximation,
	we only require the values of $\int_\X k(x, y)\dd\mu(y)$
	for $x\in\X$.
	
	Recall that $k_s^Z(x, y)$ is the rank-$s$
	Nystr{\"o}m approximation of the kernel $k$
	based on the point set $Z = (z_1, \ldots, z_\ell)$.
	From \eqref{eq:nys-sum},
	we can use $\phi_i^Z \coloneqq u_i^\top k(Z, \cdot)$
	as test functions.
	
	\begin{thm}\label{thm:nys-exact}
		Let $n \le \ell$ and $\delta > 0$,
		and let $Z$ be an $\ell$-point independent sample
		from $\mu$.
		If Algorithm~\ref{algo:rchq} with
		$k_0 = k_{n-1}^Z$ and $\phi_i = \phi_i^Z$
		successfully outputs a convex quadrature
		$Q_n$,
		then with probability at least $1-\delta$,
		we have
		\[
		\wce(Q_n;\Hil_k, \mu)^2
		\le
		4n\sigma_n
		+4\sum_{m=n+1}^\infty\sigma_m
		+
		\frac{8(n-1)\sup_{x\in\X} k(x, x)}{\sqrt{\ell}}
		\left(1 + \sqrt{2\log \frac1\delta}\right)
		\]
	\end{thm}
	\begin{proof}
		As in the proof of Corollary \ref{cor:nys},
		$k-k_{n-1}^Z$ is positive semi-definite.
		Also, we can assume $\sum_{m=1}^\infty\sigma_m<\infty$,
		as otherwise the right-hand side is infinity.
		Thus $k$ and $k_0=k_{n-1}^Z$ satisfy Assumption \ref{asp:ker}.
		
		Note that for a function of the form
		$c(x, y) = a\cdot b(x)b(y)$ with
		$a\in\R$ and $b\in L^2(\mu)$,
		and an orthonomal set $(f_i)_{i\in I}\subset L^2(\mu)$
		of $L^2(\mu)$ with $b\in\overline{\mathop\mathrm{span}
			\{f_i\mid i\in I\}}$,
		we have
		\begin{align}
			\sum_{i\in I}\iint_{\X\times\X}
			f_i(x)c(x, y)f_i(y)\dd\mu(x)\dd\mu(y)
			&= a \sum_{i\in I}\ip{b, f_i}_{L^2}^2\nonumber\\
			&= a \lVert b \rVert_{L^2}^2
			= \int_\X c(x, x)\dd\mu(x).\label{eq:trace-ip}
		\end{align}
		
		If we here use the orthonormal set $(e_m)_{m=1}^\infty$
		that appears in the Mercer decomposition,
		by letting $k_1\coloneqq k-k_{n-1}^Z$ and using
		the linear extension
		of \eqref{eq:trace-ip},
		we have
		\begin{align}
			\int_\X k_1(x, x)\dd\mu(x)
			&= \int_\X (k(x, x) - k_0(x, x))\dd\mu(x)
			\nonumber\\
			&= \sum_{m=1}^\infty \ip{e_m, (\K-\K_{n-1}^Z) e_m}_{L^2}
			\nonumber\\
			&\le \sum_{m=n}^\infty \ip{e_m, \K e_m}_{L^2} +
			\sum_{m=1}^{n-1}\lVert \K - \K_{n-1}^Z \rVert
			\lVert e_m\rVert^2_{L^2}\nonumber\\
			&= \sum_{m=n}^\infty \sigma_m + 
			(n-1)\lVert \K - \K_{n-1}^Z \rVert.
			\label{eq:nys-estimate}
		\end{align}
		Then, by combining this with Theorem \ref{thm:nys-op-norm}
		and Theorem \ref{thm:unify}(c),
		we obtain the desired estimate.
	\end{proof}
	
	\begin{remark}
		If we denote by $N_\phi$
		the required number of samples,
		the computational complexity of the above algorithm
		becomes $\ord{n\ell N_\phi + n\ell^2
			+ C(n, N_\phi)}$,
		including the cost of computing
		the Nystr{\"o}m approximation as well as
		test functions at $N_\phi$ samples
		(see also Remark \ref{rem:comp-nys}).
	\end{remark}

	\section{Additional Numerical Experiments}\label{app:experiments}
	In this section, we provide additional experiments
	on Algorithm \ref{algo:rchq} using random convex hulls,
	as well as the approximated version of the {\bf N. + emp} described in Remark \ref{rem:comp-nys}.
	Section \ref{sec:app-sob} shows the comparison of Algorithm \ref{algo:rchq}
	(with Mercer/Nystr\"om approximation)
	with some of the methods mentioned in the main text
	under the periodic Sobolev spaces with uniform measure.
	Section \ref{sec:app-ml} investigates Algorithm \ref{algo:rchq}
	(with Nyst\"om approximation)
	as well as the approximate but fast algorithm for {\bf N. + emp},
	under the setting of empirical measure reduction.
	
	\subsection{Periodic Sobolev Spaces with Uniform Measure}\label{sec:app-sob}
	We conducted experiments under the same setting as in Section \ref{sec:sob},
	except that we additionally have the following methods:
	\begin{description}
		\item[Nystr\"om, Nystr\"om + opt:]
		We used the same test functions as {\bf N. + emp}
		with the random set $Z$ of size $\ell = 10n$,
		but for Algorithm \ref{algo:rchq}.
		We used $N = 10n$ samples for the LP \eqref{eq:lp}.
		In {\bf Nystr\"om + opt}
		we additionally optimized the convex weights using \eqref{eq:wce-general}.
		\item[Mercer, Mercer + opt ($d=1$):]
		We used the same test functions as {\bf M. + emp},
		but for Algorithm \ref{algo:rchq}.
		We used $N = 10n$ samples for the LP \eqref{eq:lp}.
		In {\bf Mercer + opt}
		we additionally optimized the convex weights using \eqref{eq:wce-general}.
	\end{description}
	\begin{figure}
		\centering
		\begin{subfigure}[h]{0.49\hsize}
			\centering
			\includegraphics[width=\hsize]{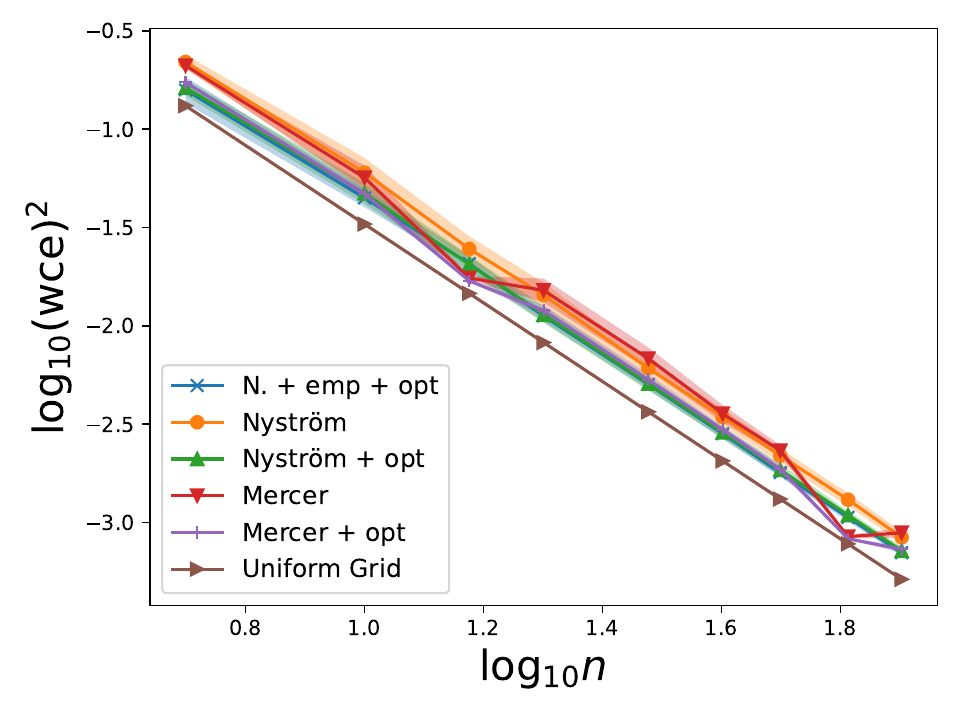}
			\caption{$d=1$, $r = 1$}
		\end{subfigure}
		\begin{subfigure}[h]{0.49\hsize}
			\centering
			\includegraphics[width=\hsize]{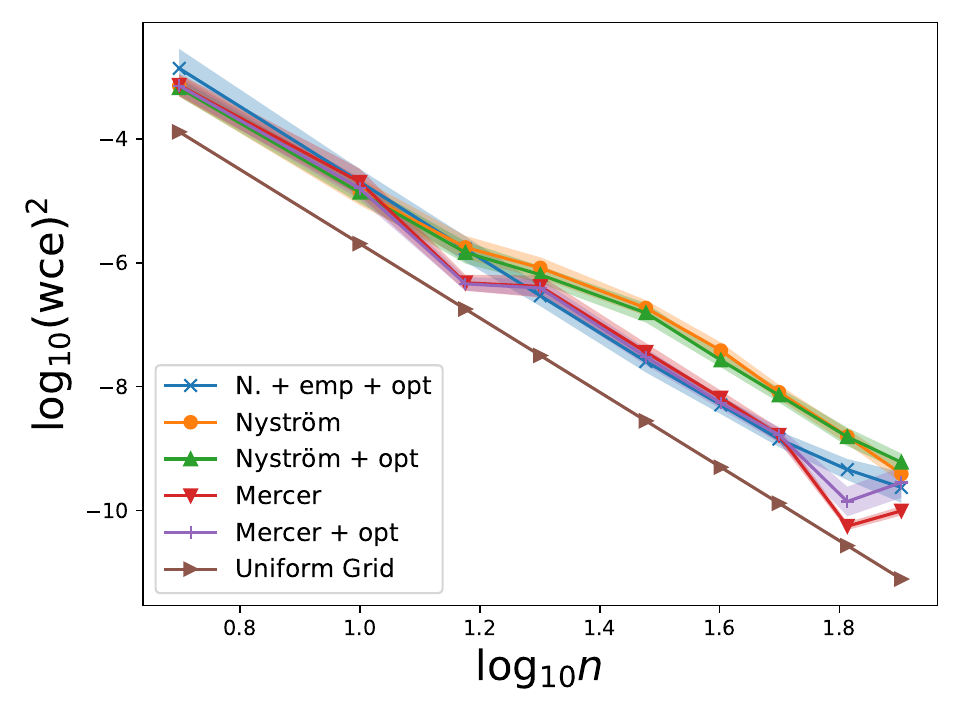}
			\caption{$d=1$, $r = 3$}
		\end{subfigure}
		\begin{subfigure}[h]{0.49\hsize}
			\centering
			\includegraphics[width=\hsize]{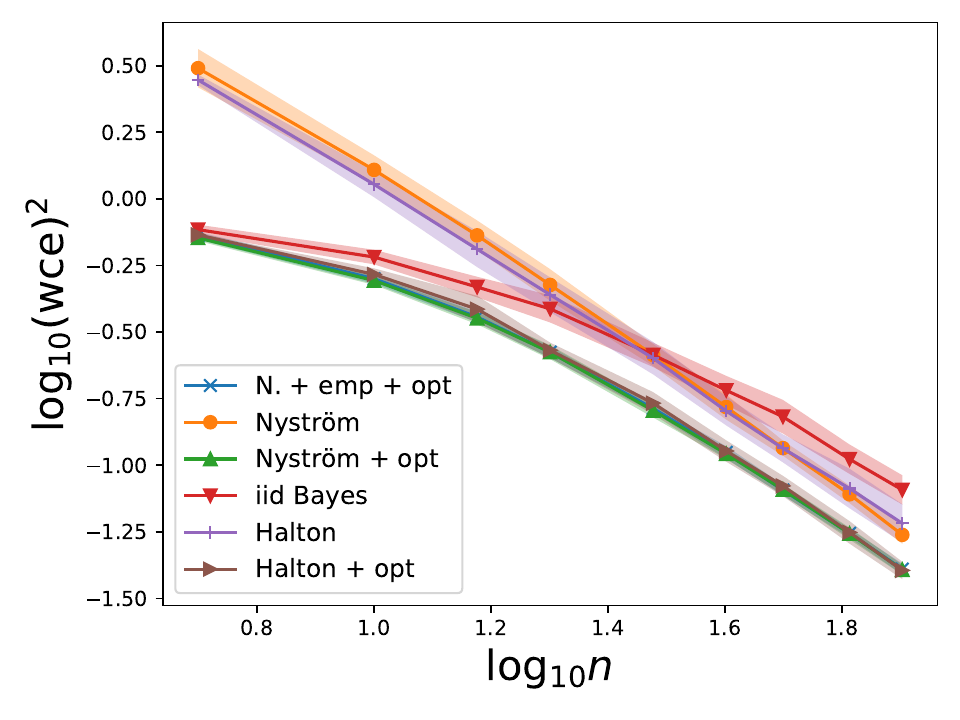}
			\caption{$d=2$, $r = 1$}
		\end{subfigure}
		\begin{subfigure}[h]{0.49\hsize}
			\centering
			\includegraphics[width=\hsize]{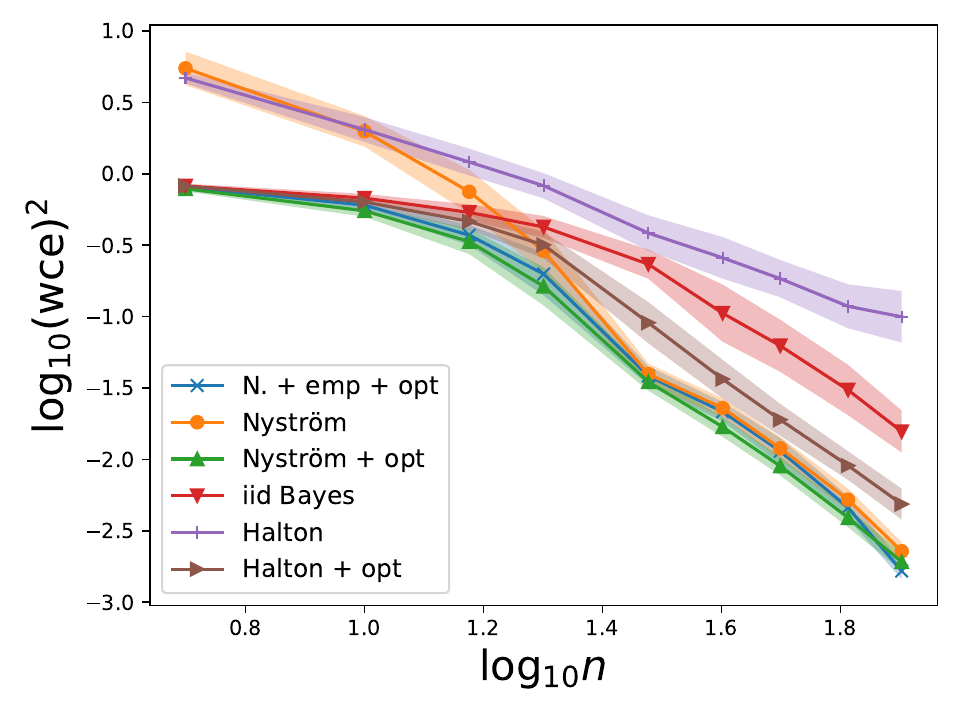}
			\caption{$d=3$, $r = 3$}
		\end{subfigure}
		\caption{Periodic Sobolev spaces with kernel $k_r^{\otimes d}$:
			The average of $\log_{10}(\wce(Q_n; \Hil_k, \mu)^2)$ over $50$ trials
			is plotted for each method of obtaining $Q_n$.
			The shaded regions are sample standard deviation.
			The worst computational time per one trial
			was 5 seconds of {\bf N. + emp} and {\bf N. + emp + opt}
			in $(d, r, n) = (3, 3, 80)$, while {\bf Nystr\"om} and {\bf Nystr\"om + opt}
			ran in 0.9 seconds under the same setting. There were no infeasible LPs.}
		\label{app-fig-sob}
	\end{figure}
	\begin{figure}
		\centering
		\begin{subfigure}[h]{0.49\hsize}
			\centering
			\includegraphics[width=\hsize]{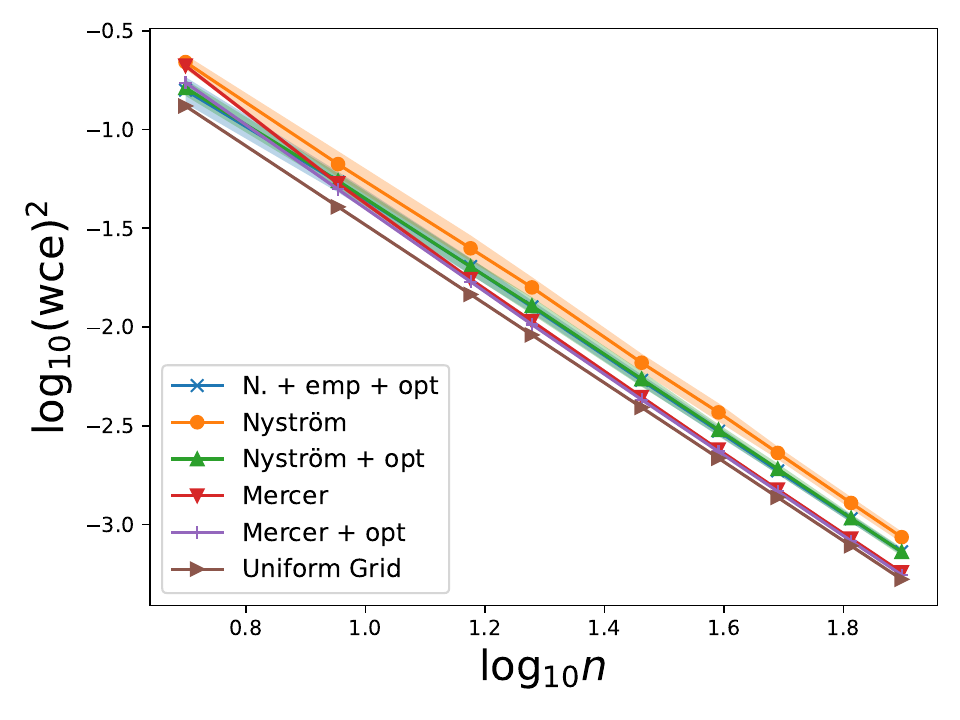}
			\caption{$d=1$, $r = 1$}
		\end{subfigure}
		\begin{subfigure}[h]{0.49\hsize}
			\centering
			\includegraphics[width=\hsize]{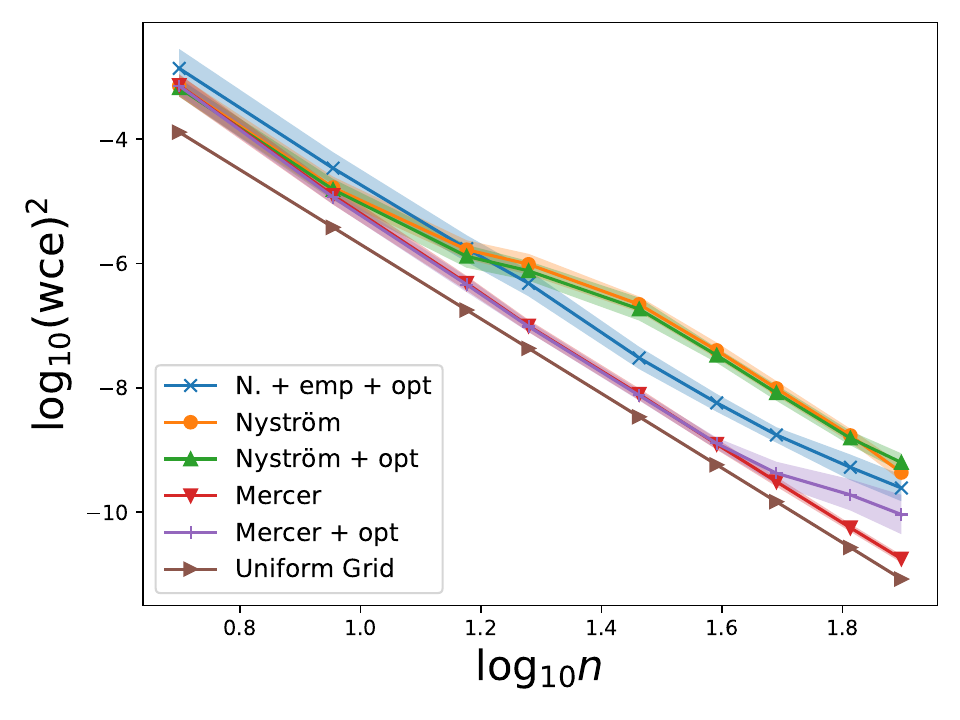}
			\caption{$d=1$, $r = 3$}
		\end{subfigure}
		\caption{Supplemental experiments for
			Figure \ref{app-fig-sob}. $n$ is all odd.}\label{app-fig-sob-supple}
	\end{figure}
	
	The results are given in Figure \ref{app-fig-sob}.
	The weights of {\bf Nystr\"om} and {\bf Mercer} are already almost
	optimized as they exactly integrate a certain family of functions,
	so the additional CQP \eqref{eq:wce-general} does not change the error
	so much. Surprising is that {\bf N. + emp + opt} is almost as good as
	{\bf Nystr\"om + opt} or even better.
	This implies that the recombination points with respect to a moderately large
	($N=n^2$ in this case) empirical measure
	can provide a good convergence rate in Bayesian quadrature \citep{hus12},
	even though the (equally weighted)
	empirical measure itself is not that close to the true measure.
	
	\paragraph{Odd behavior of `Mercer'.}
	As we can see in Figure \ref{app-fig-sob}(a,b),
	the methods based on the exact Mercer decomposition becomes
	very close to optimal when $n = 15, 65$.
	As it seemed to be caused by the parity of $n$,
	we carried out another experiment for $n\in\{5, 9, 15, 19, 29, 39, 49, 65, 79\}$
	(Figure \ref{app-fig-sob-supple}),
	then {\bf Mercer} and its optimization clearly became the best methods
	except the exact optimal {\bf Unifrom Grid}.
	It might be related to the structure of the periodic Sobolev space,
	that has, for each eigenvalue except for $1$, two-dimensional eigenspace 
	($\cos$ and $\sin$), but needs further investigation.
	Also, in the case $(d, r) = (1, 3)$,
	we see `{\bf + opt}' make the quadrature less accurate for a big $n$,
	but it is theoretically almost impossible,
	so it seems to be caused by numerical accuracy of the CQP solver.
	
	\subsection{Measure Reduction in Machine Learning Datasets}\label{sec:app-ml}
	We conducted experiments under the same setting as in Section \ref{sec:ml}.
	We additionally adopted {\bf Nystr\"om, Nystr\"om + opt} (with $N = 20n$),
	and {\bf FNE, FNE + opt},
	where {\bf FNE} (stands for `fast N. + emp')
	is the approximate algorithm for {\bf N. + emp}
	by omitting the inequality in \eqref{eq:nys-constraints}
	and using the randomized SVD \citep{hal11} (see Remark \ref{rem:comp-nys}).
	
	\begin{figure}[H]
		\centering
		\begin{subfigure}[h]{0.49\hsize}
			\centering
			\includegraphics[width=\hsize]{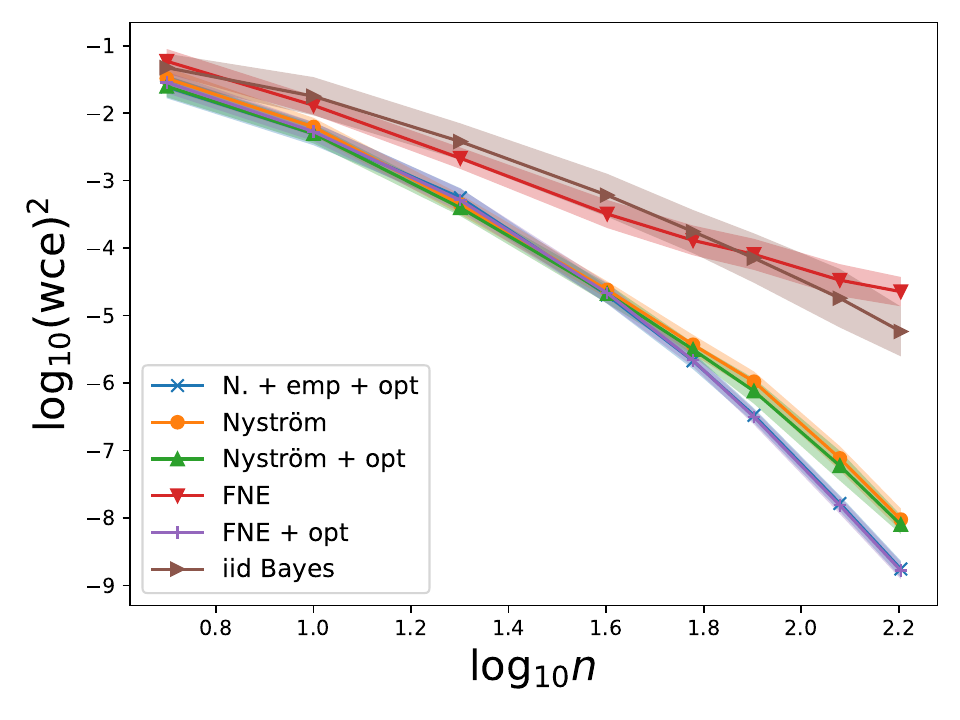}
			\caption{3D Road Network data}
		\end{subfigure}
		\begin{subfigure}[h]{0.49\hsize}
			\centering
			\includegraphics[width=\hsize]{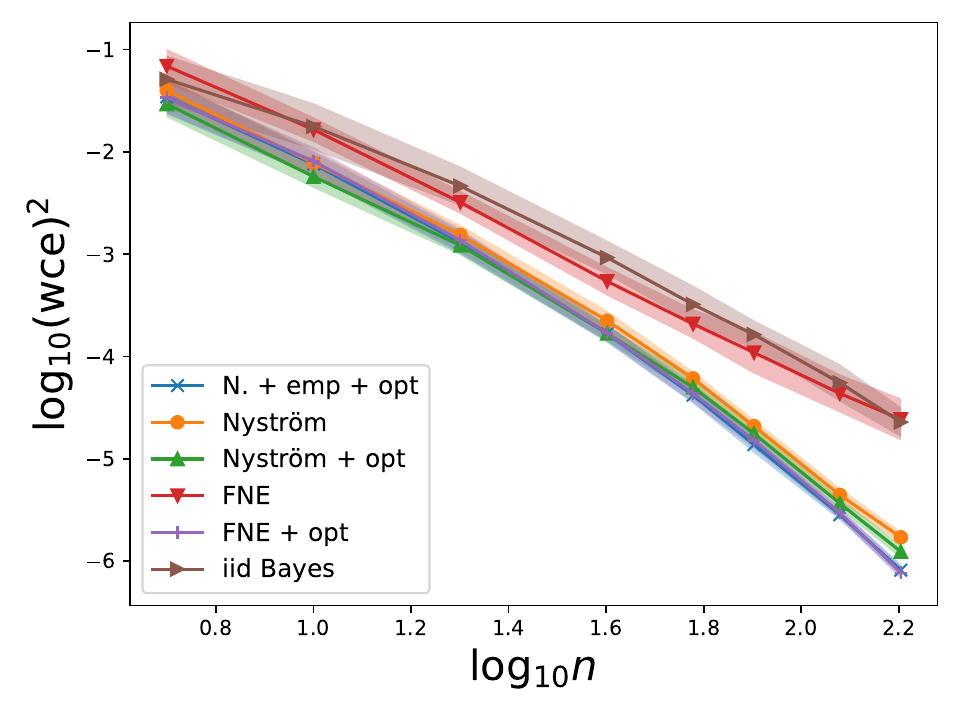}
			\caption{Power Plant data}
		\end{subfigure}
		\caption{Measure reduction in Gaussian RKHS with two ML datasets:
			The average of
			$\log_{10}(\wce(Q_n; \Hil_k, \mu)^2)$ over $50$ trials
			is plotted for each method of obtaining $Q_n$.
			The shaded regions are sample standard deviation.
			The worst computational time per one trial
			was 13 seconds of {\bf N. + emp} and {\bf N. + emp + opt}
			in 3D Road Network data with $n = 160$.
			There were 7 infeasible LPs (and 800 feasible LPs) in the experiment (a)
			with {\bf Nystr\"om} or {\bf Nystr\"om + opt}.
			There were no infeasible LPs in (b).}
		\label{app-fig-ml}
	\end{figure}
	
	The results are given in Figure \ref{app-fig-ml}.
	{\bf N. + emp + opt} and {\bf FNE + opt} show almost
	the same convergence.
	While in the largest case $n = 160$,
	the average runtime of ({\bf N. + emp + opt}, {\bf FNE + opt})
	was $(13.0, 2.07)$ seconds in 3D Road Network data
	and $(12.8, 2.09)$ seconds in Power Plant data, respectively.
	Although our theoretical guarantee no longer holds for {\bf FNE},
	it accelerates the algorithm while surprisingly maintaining the accuracy.
	{\bf Nyst\"om} or {\bf Nystr\"om + opt} behave much better than {\bf iid Bayes},
	but are slightly less accurate than {\bf N. + emp + opt} and {\bf FNE + opt},
	whereas they have good theoretical guarantees (Theorem \ref{thm:nys-exact}).
	Their computational time was basically between that of {\bf FNE + opt}
	and {\bf N. + emp + opt}.

    \paragraph{Comparison with another empirical measure.}
    The setting of `ML datasets' treated here is
    empirical measures given by some real data,
    so it is also just an approximation of a true distribution
    from the viewpoint of frequentists.
    Therefore, if we want to evaluate the performance of measure reduction methods
    with regard to the true distribution,
    we should measure the worst-case error using it.
    As it is not feasible in reality, we take another empirical measure
    $\mu^\prime$ (of the same size as
    but different from the empirical measure $\mu$, used in the construction
    of a kernel quadrature rule $Q_n$),
    and plot the quantities of $\wce(Q_n; \Hil_k, \mu^\prime)$
    to better estimate the actual performance of $Q_n$ in this section.
    
    The overall setting is the same as in Section \ref{sec:ml},
    except the following points:
    \begin{itemize}
        \item In the 3D Road Network Data Set,
        we used another random 43487-point subset
        from the remaining $434874 - 43487$ data points
        to define $\mu^\prime$.
        \item In the Combined Cycle Power Plant Data Set,
        we used exactly a half of the whole data points to define $\mu$
        (so the size of $\supp\mu$ is different from the original experiment)
        and the other half to define $\mu^\prime$.
    \end{itemize}
    Note that $\mu$ and $\mu^\prime$ were randomly taken at first
    and fixed throughout the experiment.
    The median heuristics as well as the normalization of the data
    (for both of the points in $\mu$ and $\mu^\prime$)
    was carried out by using the statistical information solely given by $\mu$.
    \begin{figure}[H]
		\centering
		\begin{subfigure}[h]{0.49\hsize}
			\centering
			\includegraphics[width=\hsize]{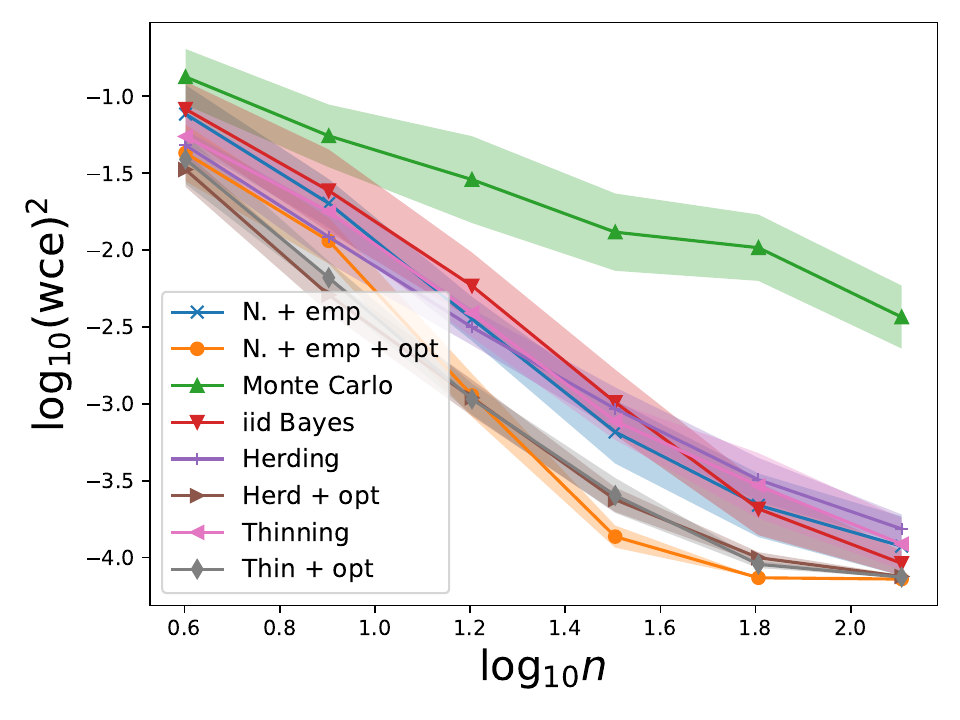}
			\caption{3D Road Network data}
		\end{subfigure}
		\begin{subfigure}[h]{0.49\hsize}
			\centering
			\includegraphics[width=\hsize]{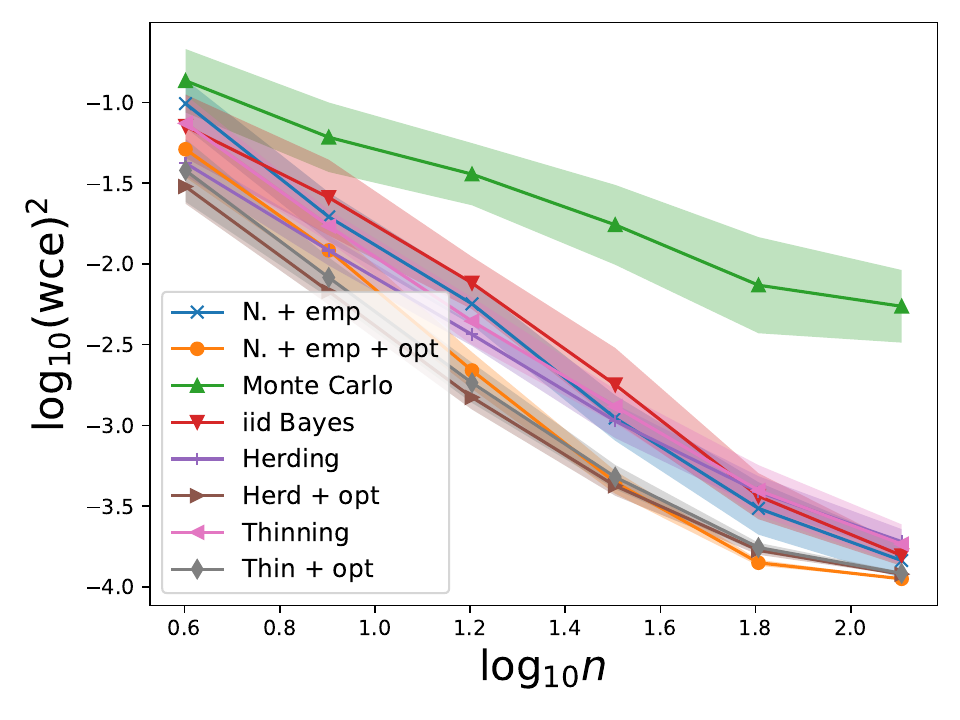}
			\caption{Power Plant data}
		\end{subfigure}
		\caption{Measure reduction in Gaussian RKHS with two ML datasets
		    with another empirical measure:
			The average of
			$\log_{10}(\wce(Q_n; \Hil_k, \mu^\prime)^2)$ over $20$ trials
			is plotted for each method of obtaining $Q_n$.
			The shaded regions show their standard deviation.
			The worst computational time per one trial
			was 14 seconds of {\bf Thinning [+ opt]}
			in Power Plant data with $n = 128$,
			where {\bf N. + emp [+ opt]} was 6.2 [6.1] seconds.}
		\label{app-fig-ml-split}
	\end{figure}
    The results are given in Figure \ref{app-fig-ml-split}.
    We can see that, though our methods are still competitive,
    the error eventually becomes dominated by the (MMD-)distance between
    $\mu$ and $\mu^\prime$ as $n$ gets larger.
    This is inevitable as we are only using the empirical measure $\mu$
    to construct $Q_n$, so in an application to this kind of setting,
    we can just pick any method whose error is sufficiently small
    compared to the `inevitable' error.
\end{document}